\documentclass[11pt]{amsart}
\usepackage{amsfonts}
\usepackage{bbm}
\usepackage{amsfonts,amssymb,amsmath,amsthm}
\usepackage{url}
\usepackage{enumerate}
\usepackage{bbm}
\usepackage[all]{xy}
\usepackage[pdftex, colorlinks, citecolor=red, backref=page]{hyperref}
\usepackage{color,xcolor}
\usepackage{pgfplots}
\usepackage{tikz}
\usepackage{setspace}
\usepackage{makecell}
\usepackage{multirow}
\usepackage{diagbox}
\newtheorem{theorem}{Theorem}[section]
\newtheorem{lemma}[theorem]{Lemma}
\newtheorem{definition}[theorem]{Definition}
\newtheorem{notion}[theorem]{}

\newtheorem{proposition}[theorem]{Proposition}
\newtheorem{corollary}[theorem]{Corollary}
\theoremstyle{definition}
\newtheorem{remark}[theorem]{Remark}

\newtheorem{example}[theorem]{Example}

\author{Qingnan An}
\address{School of Mathematics and Statistics, Northeast Normal University, Changchun, {\rm130024}, China}
\email{qingnanan1024@outlook.com}

\author{Zhichao Liu}
\address{School of Mathematical Sciences,
Dalian University of Technology,
Dalian, {\rm116024}, China }
\email{lzc.12@outlook.com}


\begin{document}

\title[Bockstein operations] {Bockstein operations and extensions with trivial boundary maps}

\begin{abstract}
In this paper, we investigate the relationship between ideal structures and the Bockstein operations in the total K-theory, offering various diagrams to demonstrate their effectiveness in classification. We  explore different situations and demonstrate a variety of conclusions, highlighting the crucial role these structures play within the framework of invariants.


\end{abstract}

\maketitle


\section{Introduction}
\subsection{Background}
The Elliott classification programme has enjoyed tremendous success in the classification of C*-algebras. This programme started with the classification of AT-algebras of real rank zero, in which Elliott classified these algebras in terms of the ordered $\mathrm{K}_{*}$-group \cite{Ell,EG}. Later, Gong \cite{G} presented an example to show that the ordered $\mathrm{K}_{*}$-group is not sufficient for the classification of non-simple AH algebras of real rank zero (and no dimension growth). Elliott, Gong, and Su \cite{EGS} constructed such examples for AD algebras by using AH algebras (indirectly). Dadarlat and Loring \cite{DL3} also gave such an example for AD algebras directly. The effort of many mathematicians \cite{DG,DL2,DL3,Ei} led to a refined invariant, the total $\mathrm{K}$-theory group
$$
\underline{\mathrm{K}}(A)=\bigoplus_{n \geq 0} \mathrm{K}_{*}(A ; \mathbb{Z}_n)
$$
endowed with a certain order structure and acted upon by the natural coefficient transformations $$\rho_n^i: \mathrm{K}_i(A) \rightarrow \mathrm{K}_i(A ; \mathbb{Z}_n),$$
$$ \kappa_{m, n}^i:  \mathrm{K}_i(A ; \mathbb{Z}_n) \rightarrow \mathrm{K}_i(A ; \mathbb{Z}_m)$$ and by the Bockstein maps $$\beta_n^i: \mathrm{K}_i(A ; \mathbb{Z}_n) \rightarrow \mathrm{K}_{i+1}(A).$$
The collection of all the maps $\rho, \beta, \kappa$ and their compositions are called Bockstein operations and denoted by $\Lambda$.

It turns out that $\underline{\mathrm{K}}(A)$ provides a better description of the connections between ideals of $A$. Eilers \cite{Ei} showed that $\underline{\mathrm{K}}(A)$ is a complete invariant for AD algebras with real rank zero and bounded torsion in $\mathrm{K}_1$. In 1997, Dadarlat and Gong \cite{DG} presented an important classification of certain approximately subhomogeneous ${\mathrm C}^*$-algebras of real rank zero and no dimension growth. (This class is called the  A$\mathcal{HD}$ class in the subsequent work, because the class in includes all real rank zero AD algebras and AH algebras with slow dimension growth.) Recently, \cite{AELL, ALZ, AL2023} push forward a classification for certain ASH algebras with one-dimensional local spectra (it includes much more general subhomogeneous building blocks than \cite{DG}).

These classifications illustrate the importance of the integration of the coefficient and the Bockstein operations along with the order structure in $\underline{\mathrm{K}}(-)$.
Our primary focus is on the relationships between Bockstein operations and the Dadalat-Gong order for ideals to explore the structure of total K-theory.

\subsection{Structure of invariant}
There are numerous reduction results indicating some intriguing interactions between the order and algebraic structure of K-groups, which make some natural maps in $\underline{\mathrm{K}}(-)$ redundant. Under certain conditions, the coefficients of total K-theory can be simplified (\cite{DE,Ei}).




In \cite{DEb}, Dadalat and Eilers constructed two non-isomorphic A$\mathcal{HD}$ $C^*$-algebras $E$ and $E'$ 
of real rank zero such that there exists a graded group isomorphism from $\underline{\mathrm{K}}(E)$ to $\underline{\mathrm{K}}(E')$ preserving the order and compatible with Bockstein operations $\rho$ and $\kappa$, but not $\beta$. 
Hence, the map $\beta$ is necessary---that is, it can not be deleted by simplification.

For the map $\rho$, they also constructed two non-isomorphic $C^*$-algebras $D$ and $D'$ of real rank zero \cite{DEb} such that there exists a graded group isomorphism from $\underline{\mathrm{K}}(D)$ to $\underline{\mathrm{K}}(D')$ compatible with Bockstein operators $\beta$ and $\kappa$.
They checked that such isomorphism can not be compatible with $\rho_k^i$ for some $k$. However, 
this isomorphism between the total K-theory doesn't preserve the Dadalat-Gong order. Inspired by their construction, we are able to  present an example to show the necessity of 
$\rho$.

Recently, it has been shown in \cite{ALcounter} that there exists two $C^*$-algebras of real rank zero and stable rank one with the same total K-theory, but they are not isomorphic. Such two algebras are not K-pure (see Definition \ref{def k pure}) and
can be distinguished by total Cuntz semigroup \cite{ALjfa,ALcounter}.  The key point is that the isomorphism between the total K-theory preserves all the Bockstein maps except the Dadalat-Gong order for ideals. Thus, we need K-pureness  to characterize this order structure.
This property indicates the extent to which the total K-theory of ideals can be preserved through embeddings.  

The Bockstein operations and the Dadalat-Gong order for ideals are also interconnected by a multitude of commutative diagrams. Combining with the techniques from extension theory, we are able to investigate the invariants between algebras and their ideals. The study of the effective roles in the total K-theory, which can deeply characterize the structure of invariants, also provides key aspects that determine the classification.

\subsection{Main results}

 In this paper, we give a detailed description on the relations between the ideal structure and the total K-theory. 
 As a byproduct, we also prove the necessity of $\rho$  for total K-theory.
\begin{definition}\rm
Denote the resulting invariant $\underline{\rm K}_{\langle\rho\rangle}(-)$ $(\underline{\rm K}_{\langle\beta\rangle}(-))$ if one deletes all the $\rho$ $(\beta)$ maps from $\underline{\rm K}(-)$. We say  a map $\eta$ is a $\underline{\rm K}_{\langle\rho\rangle}$-isomorphism if  $\eta$ is a group graded isomorphism and preserves the Bockstein operations $\beta,\kappa$ and their compositions, but not necessarily preserves the map $\rho$.

 Let $A, B$ be $C^*$-algebras. For $\star= \Lambda,\, \underline{\rm K}_{\langle\rho\rangle},\, \underline{\rm K}_{\langle\beta\rangle} $, we write
$$
(\underline{\rm K}(A),\underline{\rm K}(A)_+,\Sigma\,A)_{\star}
\cong (\underline{\rm K}(B),\underline{\rm K}(B)_+,\Sigma\,B)_\star,
$$
if there exists an ordered scaled $\star$-isomorphism.
\end{definition}
\begin{definition}\rm
Let $\mathcal{E}$ denote the class consisting of the $C^*$-algebras $E$ arising from a unital extensions of $A$ by $B$ with trivial boundary maps (see Definition \ref{trivial boundary def}), where $A, B$ are A$\mathcal{HD}$ algebras of real rank zero and stable rank one, $A$ is unital simple and $B$ is stable.

For $E_1,E_2\in\mathcal{E}$,  we denote the corresponding extensions by
$$
0\to B_i \xrightarrow{\iota_i} E_i\xrightarrow{\pi_i} A_i\to 0,\quad i=1,2.
$$
\end{definition}

Now 
we consider the following cases and demonstrate various conclusions.

Firstly, we note that for the non K-pure C*-algebras, if we assume  that the Dadalat-Gong order is preserved, together with the total K-theory, we can also get the following classification. This is the most essential part in the proof of \cite[Theorem 6.13]{ALcounter}.
\begin{theorem} \label{classification}
Let $E_1, E_2\in \mathcal{E}$, then $E_1\cong E_2$ if and only if there exist  ordered scaled  $\Lambda$-isomorphisms
$$\gamma:\,(\underline{\rm K}(B_1),\underline{\rm K}(B_1)_+,\underline{\rm K}(B_1)_+)_{\Lambda}
\cong (\underline{\rm K}(B_2),\underline{\rm K}(B_2)_+,\underline{\rm K}(B_2)_+)_{\Lambda}
$$
and
$$\eta:\, (\underline{\rm K}(E_1),\underline{\rm K}(E_1)_+,[1_{E_1}])_{\Lambda}
\cong (\underline{\rm K}(E_2),\underline{\rm K}(E_2)_+,[1_{E_2}])_{\Lambda}$$
such that the following diagram is commutative
$$
\xymatrixcolsep{2pc}
\xymatrix{
{\,\,\underline{\mathrm{K}}(B_1)\,\,} \ar[r]^-{\underline{\mathrm{K}}(\iota_1)}\ar[d]_-{\gamma}
& {\,\,\underline{\mathrm{K}}(E_1)\,\,} \ar[d]_-{ \eta}
 \\
{\,\,\underline{\mathrm{K}}(B_2) \,\,}\ar[r]_-{\underline{\mathrm{K}}(\iota_2)}
& {\,\,\underline{\mathrm{K}}(E_2)\,\,}.}
$$
\end{theorem}
Under this circumstance, the total K-theory along with the compatible ideal structure will recover the total Cuntz semigroup and can be used as invariants to classify extensions (see \cite[5.1]{ERZ}).

Secondly, if we only assume that the diagram  above only commutes at the level of ${\rm K}_*$-group, the conclusion is still hold. In this case, $\gamma,\eta$ preserve Bockstein operations and Dadalat-Gong order, but we don't need to assume their Bockstein operations and  Dadalat-Gong  order are compatiable at the level of $\underline{\mathrm{K}}$-groups. 
\begin{theorem}\label{weak class}
Let $E_1, E_2\in \mathcal{E}$, then $E_1\cong E_2$ if and only if there exist 
$\Lambda$-isomorphisms
$$\gamma:\,(\underline{\rm K}(B_1),\underline{\rm K}(B_1)_+,\underline{\rm K}(B_1)_+)_\Lambda
\cong (\underline{\rm K}(B_2),\underline{\rm K}(B_2)_+,\underline{\rm K}(B_2)_+)_\Lambda
$$
and
$$\eta:\, (\underline{\rm K}(E_1),\underline{\rm K}(E_1)_+,[1_{E_1}])_\Lambda
\cong (\underline{\rm K}(E_2),\underline{\rm K}(E_2)_+,[1_{E_2}])_\Lambda$$
such that the following diagram is commutative
$$
\xymatrixcolsep{2pc}
\xymatrix{
{\,\,{\mathrm{K}}_*(B_1)\,\,} \ar[r]^-{{\mathrm{K}}_*(\iota_1)}\ar[d]_-{\gamma_*}
& {\,\,{\mathrm{K}}_*(E_1)\,\,} \ar[d]_-{ \eta_*}
 \\
{\,\,{\mathrm{K}}_*(B_2) \,\,}\ar[r]_-{{\mathrm{K}}_*(\iota_2)}
& {\,\,{\mathrm{K}}_*(E_2)\,\,},}
$$
where $\gamma_*, \eta_*$ is the restriction maps of $\gamma, \eta$.
\end{theorem}

Thirdly, we show that there exist two non-isomorphic C*-algebras of real rank zero and stable rank one with isomorphic scaled ordered total K-theory as $\Lambda$-module, though there is also a graded scaled ordered group isomorphism between their unique maximal ideals' total K-theory which is even compatible with the Bockstein operations $\beta,\kappa$, but not $\rho$. That is, the map $\rho$ on ideals is necessary!



\begin{theorem}\label{rho is nece}
There are two non-isomorphic $C^*$-algebras $E_1, E_2$ in $\mathcal{E}$ satisfying that there exist 
 isomorphisms
$$\gamma:\,(\underline{\rm K}(B_1),\underline{\rm K}(B_1)_+,\underline{\rm K}(B_1)_+)_{\underline{\rm K}_{\langle\rho\rangle}}
\cong (\underline{\rm K}(B_2),\underline{\rm K}(B_2)_+,\underline{\rm K}(B_2)_+)_{\underline{\rm K}_{\langle\rho\rangle}}
$$
and
$$\eta:\, (\underline{\rm K}(E_1),\underline{\rm K}(E_1)_+,[1_{E_1}])_{\Lambda}
\cong (\underline{\rm K}(E_2),\underline{\rm K}(E_2)_+,[1_{E_2}])_{\Lambda}$$
such that the following diagram is commutative
$$
\xymatrixcolsep{2pc}
\xymatrix{
{\,\,\underline{\mathrm{K}}(B_1)\,\,} \ar[r]^-{\underline{\mathrm{K}}(\iota_1)}\ar[d]_-{\gamma}
& {\,\,\underline{\mathrm{K}}(E_1)\,\,} \ar[d]_-{ \eta}
 \\
{\,\,\underline{\mathrm{K}}(B_2) \,\,}\ar[r]_-{\underline{\mathrm{K}}(\iota_2)}
& {\,\,\underline{\mathrm{K}}(E_2)\,\,}.}
$$
\end{theorem}

Finally, we show that if we do not require $\gamma$ to preserve the Bockstein operations $\beta$ on the ideals, and if  $\eta$ preserves $\beta$ on the algebras and the diagram commutes at the level of $\underline{\rm K}$-groups, then $\gamma$ preserves $\beta$ automatically. This means we don't need to check the compatibility of $\beta$ between the ideals in this case.
\begin{theorem}\label{beta not ness}
Let $E_1, E_2\in \mathcal{E}$, then $E_1\cong E_2$ if and only if there exist 
 isomorphisms
$$\gamma:\,(\underline{\rm K}(B_1),\underline{\rm K}(B_1)_+,\underline{\rm K}(B_1)_+)_{\underline{\rm K}_{\langle\beta\rangle}}
\cong (\underline{\rm K}(B_2),\underline{\rm K}(B_2)_+,\underline{\rm K}(B_2)_+)_{\underline{\rm K}_{\langle\beta\rangle}}
$$
and
$$\eta:\, (\underline{\rm K}(E_1),\underline{\rm K}(E_1)_+,[1_{E_1}])_{\Lambda}
\cong (\underline{\rm K}(E_2),\underline{\rm K}(E_2)_+,[1_{E_2}])_{\Lambda}$$
such that the following diagram is commutative
$$
\xymatrixcolsep{2pc}
\xymatrix{
{\,\,\underline{\mathrm{K}}(B_1)\,\,} \ar[r]^-{\underline{\mathrm{K}}(\iota_1)}\ar[d]_-{\gamma}
& {\,\,\underline{\mathrm{K}}(E_1)\,\,} \ar[d]_-{ \eta}
 \\
{\,\,\underline{\mathrm{K}}(B_2) \,\,}\ar[r]_-{\underline{\mathrm{K}}(\iota_2)}
& {\,\,\underline{\mathrm{K}}(E_2)\,\,}.}
$$
\end{theorem}

Moreover, we point out that \cite[Theorem 3.10]{ALcounter}
also has the following formulation, which reveals that
the Dadalat-Gong order on the total K-groups of ideals is necessary, where the existence of $\gamma$ just comes from the KK-equivalence in UCT \cite{RS} and UMCT \cite{DL2}:
\begin{theorem}\label{order is nece}
There exist two non-isomorphic $C^*$-algebras $E_1, E_2$ in $\mathcal{E}$ satisfying that
there exist a $\Lambda$-isomorphism which doesn't preserve the order 
$$\gamma:\,\underline{\rm K}(B_1)
\cong \underline{\rm K}(B_2),
$$ and an ordered scaled $\Lambda$-isomorphism
$$\eta:\, (\underline{\rm K}(E_1),\underline{\rm K}(E_1)_+,[1_{E_1}])_{\Lambda}
\cong (\underline{\rm K}(E_2),\underline{\rm K}(E_2)_+,[1_{E_2}])_{\Lambda}$$
such that the following diagram is commutative
$$
\xymatrixcolsep{2pc}
\xymatrix{
{\,\,\underline{\mathrm{K}}(B_1)\,\,} \ar[r]^-{\underline{\mathrm{K}}(\iota_1)}\ar[d]_-{\gamma}
& {\,\,\underline{\mathrm{K}}(E_1)\,\,} \ar[d]_-{ \eta}
 \\
{\,\,\underline{\mathrm{K}}(B_2) \,\,}\ar[r]_-{\underline{\mathrm{K}}(\iota_2)}
& {\,\,\underline{\mathrm{K}}(E_2)\,\,}.}
$$
Here, we even have the restriction map is an ordered scaled isomorphism
$$\gamma_*:\,({\rm K}_*(B_1),{\rm K}_*^+(B_1),{\rm K}_*^+(B_1))
\cong ({\rm K}_*(B_2),{\rm K}_*^+(B_2),{\rm K}_*^+(B_2)).
$$
\end{theorem}

Note that  Theorem \ref{weak class}--\ref{order is nece} are independent perturbation results of Theorem \ref{classification}. One may find that the only reasonable combination of perturbations we need to consider is as follows.
\begin{theorem}\label{beta+com}
There exist two non-isomorphic $C^*$-algebras $E_1, E_2$ in $\mathcal{E}$ satisfying that there exist  isomorphisms 
$$\gamma:\,(\underline{\rm K}(B_1),\underline{\rm K}(B_1)_+,\underline{\rm K}(B_1)_+)_{\underline{\rm K}_{\langle\beta\rangle}}
\cong (\underline{\rm K}(B_2),\underline{\rm K}(B_2)_+,\underline{\rm K}(B_2)_+)_{\underline{\rm K}_{\langle\beta\rangle}}
$$
and
$$\eta:\, (\underline{\rm K}(E_1),\underline{\rm K}(E_1)_+,[1_{E_1}])_{\Lambda}
\cong  (\underline{\rm K}(E_2),\underline{\rm K}(E_2)_+,[1_{E_2}])_{\Lambda}$$
such that the following diagram is commutative
$$
\xymatrixcolsep{2pc}
\xymatrix{
{\,\,{\mathrm{K}}_*(B_1)\,\,} \ar[r]^-{{\mathrm{K}}_*(\iota_1)}\ar[d]_-{\gamma_*}
& {\,\,{\mathrm{K}}_*(E_1)\,\,} \ar[d]_-{ \eta_*}
 \\
{\,\,{\mathrm{K}}_*(B_2) \,\,}\ar[r]_-{{\mathrm{K}}_*(\iota_2)}
& {\,\,{\mathrm{K}}_*(E_2)\,\,},}
$$
where $\gamma_*, \eta_*$ is the restriction maps of $\gamma, \eta$.
\end{theorem}
So Theorem \ref{weak class}, Theorem \ref{beta not ness} and Theorem \ref{beta+com} show that
the commutativity of total K-theory and the map $\beta$ form some common
support for the classification. That is,
if we have one of them, the framing is complete; but if we destroy the both, the classification will collapse.

Now we assemble all the results on the effectiveness of invariants into the following diagram:
\begin{table}[h]
  \centering
\begin{tabular}[h!]{|c|c|c|c|c|c|} \hline 
 \multirow{2}{*}{\diagbox{Results}{Cases}} & \multicolumn{3}{|c|}{Total K-theory of ideals} & Diagram & \multirow{2}{*}{Isomorphic}  \\  \cline{2-4}
 & $\,\,\,\,\,\,\rho\,\,\,\,\,\,$ &$\,\,\,\,\,\,\beta\,\,\,\,\,\,$ & Order &  commutativity &   \\ \hline
  Theorem \ref{classification} & \makecell[c]{Yes} & Yes & \makecell[c]{Yes} & $\underline{\mathrm{K}}$ & \makecell[c]{Yes}  \\ \hline
  Theorem \ref{weak class} & Yes & Yes & \makecell[c]{Yes}& \,\,$\mathrm{K}_*$ & \makecell[c]{Yes} \\ \hline
  Theorem \ref{rho is nece} & No& Yes &\makecell[c]{Yes} & $\underline{\mathrm{K}}$ &\makecell[c]{No} \\ \hline
   Theorem \ref{beta not ness} & Yes & -- & \makecell[c]{Yes}& $\underline{\mathrm{K}}$ & \makecell[c]{Yes} \\ \hline
    Theorem \ref{order is nece} & Yes & Yes & \makecell[c]{No}& $\underline{\mathrm{K}}$ & \makecell[c]{No} \\ \hline
    Theorem \ref{beta+com} & Yes & No & \makecell[c]{Yes} & $\,\,{\mathrm{K}_*}$ & \makecell[c]{No} \\ \hline
\end{tabular}
\end{table}

Here, we point out that our Theorem \ref{beta aut thm} implies that the $\underline{\rm K}_{\langle\beta\rangle}$-isomorphism in Theorem \ref{beta not ness}
is exactly an ordered $\Lambda$-isomorphism. Hence, in fact, Theorem \ref{beta not ness}
shares the same circumstance of Theorem \ref{classification}.


This series of results forms a discussion of necessities of the latticed isomorphisms in both \cite[Definition 6.11]{ALcounter} and \cite[Definition 5.6]{ALL}.

\section{Preliminaries}

\subsection{Terminologies}

\begin{notion}\rm
  Let $A$ be a unital $\mathrm{C}^*$-algebra. $A$ is said to have stable rank one, written $sr(A)=1$, if the set of invertible elements of $A$ is dense. $A$ is said to have real rank zero, written $rr(A)=0$, if the set of invertible self-adjoint elements is dense in the set $A_{sa}$ of self-adjoint elements of $A$. If $A$ is not unital, let us denote the unitization of $A$ by $\widetilde{A}$. A non-unital $\mathrm{C}^*$-algebra is said to have stable rank one (or real rank zero) if its unitization has stable rank one (or real rank zero).
\end{notion}

\begin{notion}\rm \label{AHD}{\rm (}\cite[2.2-2.3]{GJL}, \cite{GJL2}{\rm )}
We shall say a ${\rm C}^*$-algebra is an A$\mathcal{HD}$ algebra,
if it is an inductive limit of finite direct sums of algebras $M_n(\widetilde{\mathbb{I}}_p)$ and $PM_n(C(X))P$, where $$
\mathbb{I}_p=\{f\in M_p(C_0(0,1]):\,f(1)=\lambda\cdot1_p,\,1_p {\rm \,is\, the\, identity\, of}\, M_p\}
$$ is the Elliott-Thomsen dimension drop interval algebra
and $X$ is one of the following finite connected CW complexes: $\{pt\},~\mathbb{T},~[0, 1],~W_k.$ Here, $P\in M_n(C(X))$ is a projection and  for $k\geq 2$, $W_k$ denotes the Moore space obtained by attaching the unit disk to the circle by a degree $k$-map, such as $f_k: \mathbb{T}\rightarrow \mathbb{T}$, $e^{{\rm i}\theta}\mapsto e^{{\rm i}k\theta}$.  (This class of ${\rm C}^*$-algebras also play an important role in the classification of ${\rm C}^*$-algebras of ideal property \cite{GJLP1,GJLP2}.)

\end{notion}

\subsection{The total K-theory}

\begin{notion}\label{def k-total}\rm
 (\cite[Section 4]{DG}) For $n\geq 2$, the mod-$n$ K-theory groups are defined by
$$\mathrm{K}_* (A;\mathbb{Z}_n)=\mathrm{K}_*(A\otimes C_0(W_n)),$$
where $\mathbb{Z}_n:=\mathbb{Z}/n\mathbb{Z}$, $W_n$ denotes the Moore space obtained by attaching the unit disk to the circle by a degree $n$-map. The ${\rm C}^*$-algebra $C_0(W_n)$ of continuous functions vanishing at the base point is isomorphic to the mapping cone of the canonical map of degree $n$ from $C(\mathbb{T})$ to itself. For $n=0$, we set ${\rm K}_{*}(A ; \mathbb{Z}_n)=$ ${\rm K}_{*}(A)$ and for $n=1, {\rm K}_{*}(A ; \mathbb{Z}_n)=0$.

For a $\mathrm{C}^{*}$-algebra $A$, one defines the total K-theory of $A$ by
$$
\underline{{\rm K}}(A)=\bigoplus_{n=0}^{\infty} {\rm K}_{*}(A ; \mathbb{Z}_n) .
$$
It is a $\mathbb{Z}_2 \times \mathbb{Z}^{+}$graded group. For $m\in \mathbb{Z}$, denote $[m]_n=:m+n\mathbb{Z}\in\mathbb{Z}_n.$ It was shown in \cite{S} that the coefficient maps
$$
\begin{gathered}
\rho: \mathbb{Z} \rightarrow \mathbb{Z}_n, \quad \rho(1)=[1]_n, \\
\kappa_{mn, m}: \mathbb{Z}_m \rightarrow \mathbb{Z}_{mn}, \quad \kappa_{m n, m}([1]_m)=n[1]_{mn}, \\
\kappa_{n, m n}: \mathbb{Z}_{mn} \rightarrow \mathbb{Z}_n, \quad \kappa_{n, m n}([1]_{mn})=[1]_{n},
\end{gathered}
$$
induce natural transformations
$$
\rho_{n}^{j}: {\rm K}_{j}(A) \rightarrow {\rm K}_{j}(A ; \mathbb{Z}_n),
$$
$$
\kappa_{m n, m}^{j}: {\rm K}_{j}(A ; \mathbb{Z}_m) \rightarrow {\rm K}_{j}(A ; \mathbb{Z}_{mn}),
$$
$$
\kappa_{n, m n}^{j}: {\rm K}_{j}(A ; \mathbb{Z}_{mn}) \rightarrow {\rm K}_{j}(A ; \mathbb{Z}_n) .
$$
The Bockstein operation
$$
\beta_{n}^{j}: {\rm K}_{j}(A ; \mathbb{Z}_n) \rightarrow {\rm K}_{j+1}(A)
$$
appears in the six-term exact sequence
$$
{\rm K}_{j}(A) \stackrel{\times n}{\longrightarrow} {\rm K}_{j}(A) \stackrel{\rho_{n}^{j}}{\longrightarrow} {\rm K}_{j}(A ; \mathbb{Z}_n) \stackrel{\beta_{n}^{j}}{\longrightarrow} {\rm K}_{j+1}(A) \stackrel{\times n}{\longrightarrow} {\rm K}_{j+1}(A)
$$
induced by the cofibre sequence
$$
A \otimes S C_{0}\left(\mathbb{T}\right) \longrightarrow A \otimes C_{0}\left(W_{n}\right) \stackrel{\beta}{\longrightarrow} A \otimes C_{0}\left(\mathbb{T}\right) \stackrel{n}{\longrightarrow} A \otimes C_{0}\left(\mathbb{T}\right),
$$
where $S C_{0}\left(\mathbb{T}\right)$ is the suspension algebra of $C_{0}\left(\mathbb{T}\right)$.


The collection of all the transformations $\rho, \beta, \kappa$ and their compositions is denoted by $\Lambda$. $\Lambda$ can be regarded as the set of morphisms in a category whose objects are the elements of $\mathbb{Z}_2 \times \mathbb{Z}^{+}$.
 Abusing the terminology, $\Lambda$ will be called the category of Bockstein operations. Via the Bockstein operations, $\underline{{\rm K}}(A)$ becomes a $\Lambda$-module. It is natural to consider the group $\operatorname{Hom}_{\Lambda}(\underline{{\rm K}}(A), \underline{{\rm K}}(B))$ consisting of all $\mathbb{Z}_2 \times \mathbb{Z}^{+}$ graded group morphisms which are $\Lambda$-linear, i.e. preserve the action of the category $\Lambda$.



For any $n\in \mathbb{N},\,j=0,1$, if $\psi:\, A\to B$ is a homomorphism, we will denote
 $\underline{\rm K}(\psi):\, \underline{{\rm K}}(A)\to\underline{{\rm K}}(B)$ and ${\rm K}_j (\psi;\mathbb{Z}_n):\, {{\rm K}}_j(A;\mathbb{Z}_n)\to{{\rm K}}_j(B;\mathbb{Z}_n)$ to be the induced maps.

 If $\lambda:\,\underline{{\rm K}}(A)\to\underline{{\rm K}}(B)$ is a graded map, we will denote
$$
\lambda_n^j:\,{\rm K}_j(A;\mathbb{Z}_n)\to{\rm K}_j(B;\mathbb{Z}_n),\,\,n\in \mathbb{N},\,j=0,1
$$
to be the restriction maps.

\end{notion}

\begin{notion}\label{dg order}{\bf (Dadarlat-Gong order) ~}\rm
 Let  $A$ be  a separable C*-algebra of stable rank one. Denote by $\mathcal{K}$  the compact operators on a separable infinite-dimensional Hilbert space.
For any $[e]\in \mathrm{K}_0^+(A)$, denote
$I_e$  the ideal of $A\otimes\mathcal{K}$ generated by $e$ and
denote $\underline{\mathrm{K}}({I_e}\,|\,A)$ to be the image of $\underline{\mathrm{K}}({I_e})$ in $\underline{\mathrm{K}}(A)$, i.e.,
$$\underline{\mathrm{K}}({I_e}\,|\,A)=:
\underline{\mathrm{K}}(\iota_{I_e})(\underline{\mathrm{K}}({I_e}))\subset \underline{\mathrm{K}}(A),$$
where $\iota_{I_e}:\,I_e\to A\otimes\mathcal{K} $ is the natural embedding map. The following is a positive cone for total K-theory of $A$ (\cite[Definition 4.6]{DG}):
$$\textstyle
\underline{\mathrm{K}}(A)_+=\{([e],\mathfrak{u},
\bigoplus\limits_{n=1}^{\infty}(\mathfrak{s}_{n,0},\mathfrak{s}_{n,1})):[e]\in \mathrm{K}_0^+(A),([e],\mathfrak{u},\bigoplus\limits_{n=1}^{\infty}
(\mathfrak{s}_{n,0},\mathfrak{s}_{n,1}))\in \underline{\mathrm{K}}({I_e}\,|\, A)\},
$$
where $\mathfrak{u}\in {\rm K}_1(A)$, $\mathfrak{s}_{n,0}\in {\rm K}_0(A;\mathbb{Z}_n)$ and  $\mathfrak{s}_{n,1}\in {\rm K}_1(A;\mathbb{Z}_n)$ are equivalent classes. We will call this order structure by Dadarlat-Gong order.

In particular, we may also denote
$$\mathrm{K}_*^+(A)=:\mathrm{K}_*(A)\cap \underline{\mathrm{K}}(A)_+,$$
where $\mathrm{K}_*(A)$ is identified with its natural image in $\underline{\mathrm{K}}(A)$.
\end{notion}

\begin{theorem}{\rm (\cite[Porposition 4.8--4.9]{DG})}\label{ordertotal}
Suppose that $A$ is of stable rank one and has an approximate unit $(e_n)$ consisting
of projections. Then

{\rm (i)} $\underline{\mathrm{K}}(A)=\underline{\mathrm{K}}(A)_+-\underline{\mathrm{K}}(A)_+$;

{\rm (ii) } $\underline{\mathrm{K}}(A)_+\cap\{-\underline{\mathrm{K}}(A)_+\} = \{0\}$, and hence,
$(\underline{\mathrm{K}}(A),\underline{\mathrm{K}}(A)_+)$ is an ordered group;

{\rm (iii)} For any $x\in \underline{\mathrm{K}}(A)$, there are positive integers $k$, $n$ such that $k[e_n]+x \in \underline{\mathrm{K}}(A)_+$.
\end{theorem}

Denote by $\mathcal{N}$ the ``bootstrap'' category of \cite{RS}.
The following is immediately from \ref{dg order}:
\begin{proposition}\label{simple k* to k total}
Let $A_1,A_2$ be simple algebras of stable rank one in $\mathcal{N}$. Then for any
$$
\alpha^*:\,(\mathrm{K}_*(A_1),\mathrm{K}_*^+(A_1))\cong (\mathrm{K}_*(A_2),\mathrm{K}_*^+(A_2)),
$$
there exists an
$$
\underline{\alpha}:\,(\underline{\mathrm{K}}(A_1),\underline{\mathrm{K}}(A_1)_+)\cong (\underline{\mathrm{K}}(A_2),\underline{\mathrm{K}}(A_2)_+)
$$
whose restriction is $\alpha^*$.
\end{proposition}

\subsection{Extension and K-pureness}

\begin{definition}\label{trivial boundary def}\rm
Let
$
e: 0 \rightarrow B \rightarrow E \rightarrow A \rightarrow 0
$
be an extension  of $\mathrm{C}^{*}$-algebras.
Denote by $\delta_{j}: \mathrm{K}_{j}(A) \rightarrow \mathrm{K}_{1-j}(B)$, $j=0,1$, the boundary maps of  six-term exact sequence.
We say $e$ has $trivial$ $boundary$ $maps$, if both $\delta_0$ and $\delta_1$ are the zero map.

Suppose both $A,B$ are of stable rank one and real
rank zero; by \cite[Proposition 4]{LR}, $E$ has stable rank one and real
rank zero if, and only if, $e$ has trivial boundary maps. Thus, if $E\in \mathcal{E}$,  then $sr(E)=1$ and $rr(E)=0$. 
\end{definition}
\begin{definition} {\rm (}\cite[Lemma 4.1]{ALpams}{\rm )}\label{def k pure}\rm
Let us say that a ${\rm C}^*$-algebra $A$ is K-$pure$, if for any ideal $I$ of $A$,
the natural map from $\underline{\mathrm{K}}(I)$ to $\underline{\mathrm{K}}(A)$ is always injective.
\end{definition}




It has been shown in \cite[Proposition 4.4]{DE} that all  A$\mathcal{HD}$ algebras of real rank zero are K-pure.


\begin{definition}\rm
Let $A$, $B$ be $C^*$-algebras and
$e: 0 \to B \to E \to A \to 0 $
be an extension of $A$ by $B$ with Busby invariant $\tau:\,A\to Q(B)$.
We say extension $e$ is essential, if $\tau$ is injective.
If $A$ is unital, we say extension $e$ is unital, if $\tau$ is unital.
\end{definition}
\begin{definition}\rm
Let
$e_i:0 \to B \to E_i \to A \to 0$ ($i=1,2$)
be two
extensions of $A$ by $B$ with Busby invariants $\tau_i$ for $i = 1, 2$.  We say  $e_1$ and $e_2$ are  $strongly$ $unitarily$ $equivalent$, denoted by $e_1
\sim_s e_2$,
if there exists a unitary $u \in M(B)$ such that $\tau_2(a) = \pi(u)\tau_1(a)\pi(u)^*$ for all $a \in A$,
where $\pi:\,M(B)\to M(B)/B.$ It is well-known that $e_1
\sim_s e_2$ implies $E_1\cong E_2$.


\end{definition}


Denote  by ${\rm Text}_{s}^u(A,B)$ the set of strongly unitary equivalence classes of all the unital extensions of $A$ by $B$ with trivial boundary maps.
\begin{theorem}[Theorem 3.5 of \cite{ALpams}]\label{strong wei}
Let $A,B$ be nuclear separable ${\rm C}^*$-algebras of stable rank one and real rank zero with $A\in \mathcal{N}$. 
Assume that  $A$ is unital simple,  $B$ is stable and $({\rm K}_0(B),{\rm K}_0^+(B))$  is weakly unperforated.
Then  we have
$$
\mathrm{Ext}_{[1]}(\mathrm{K}_*(A), \mathrm{K}_*(B))\cong {\rm Text}_{s}^u(A,B).
$$
\end{theorem}


\section{Proof of Theorem 1.4}

In order to demonstrate Theorem \ref{classification}, it is sufficient to establish its stronger iteration, namely, Theorem \ref{weak class}.

\begin{proof}[Proof of Theorem \ref{weak class}]
Suppose there exist $\Lambda$-isomorphisms
$$\gamma:\,(\underline{\rm K}(B_1),\underline{\rm K}(B_1)_+,\underline{\rm K}(B_1)_+)_{\Lambda}
\cong (\underline{\rm K}(B_2),\underline{\rm K}(B_2)_+,\underline{\rm K}(B_2)_+)_{\Lambda}
$$
and
$$\eta:\, (\underline{\rm K}(E_1),\underline{\rm K}(E_1)_+,[1_{E_1}])_{\Lambda}
\cong (\underline{\rm K}(E_2),\underline{\rm K}(E_2)_+,[1_{E_2}])_{\Lambda}$$
such that the following diagram is commutative
$$
\xymatrixcolsep{2pc}
\xymatrix{
{\,\,{\mathrm{K}}_*(B_1)\,\,} \ar[r]^-{{\mathrm{K}}_*(\iota_1)}\ar[d]_-{\gamma_*}
& {\,\,{\mathrm{K}}_*(E_1)\,\,} \ar[d]_-{ \eta_*}
 \\
{\,\,{\mathrm{K}}_*(B_2) \,\,}\ar[r]_-{{\mathrm{K}}_*(\iota_2)}
& {\,\,{\mathrm{K}}_*(E_2)\,\,}.}
$$

Since $B_1$ and $B_2$ are A$\mathcal{HD}$ algebras of real rank zero, there exists an isomorphism
$\phi_0:\,B_1\to B_2$ such that $ \underline{\mathrm{K}}(\phi_0)=\gamma$. By assumption,
we have the following diagram (not necessarily commutative):
$$\xymatrixcolsep{2pc}
\xymatrix{
 {\,\,\underline{\mathrm{K}}(B_1)\,\,} \ar[d]_-{\underline{\mathrm{K}}(\phi_0)}\ar[r]^-{\mathrm{K}_0(\iota_1)}
& {\,\,(\underline{\mathrm{K}}(E_1),[1_{E_1}])\,\,} \ar[d]_-{\eta} \ar[r]^-{\mathrm{K}_0(\pi_1)}
& {\,\,(\underline{\mathrm{K}}(A_1),[1_{A_1}])\,\,}\\
{\,\,\underline{\mathrm{K}}(B_2)\,\,} \ar[r]_-{\mathrm{K}_0(\iota_2)}
& {\,\,(\underline{\mathrm{K}}(E_2),[1_{E_2}]) \,\,} \ar[r]_-{\mathrm{K}_0(\pi_2)}
& {\,\,(\underline{\mathrm{K}}(A_2),[1_{A_2}]) \,\,} .}
$$
Restricting to $\mathrm{K}_*$, the following diagram is commutative with exact rows:
$$
\xymatrixcolsep{2pc}
\xymatrix{
{\,\,0\,\,} \ar[r]^-{}
& {\,\,\mathrm{K}_*(B_1)\,\,} \ar[d]_-{\mathrm{K}_*(\phi_0)} \ar[r]^-{\mathrm{K}_*(\iota_1)}
& {\,\,(\mathrm{K}_*(E_1),[1_{E_1}])\,\,} \ar[d]_-{\eta_*} \ar[r]^-{\mathrm{K}_*(\pi_1)}
& {\,\,(\mathrm{K}_*(A_1),[1_{A_1}])\,\,} \ar[d]_-{\varrho} \ar[r]^-{}
& {\,\,0\,\,} \\
{\,\,0\,\,} \ar[r]^-{}
& {\,\,\mathrm{K}_*(B_2)\,\,} \ar[r]_-{  \mathrm{K}_*(\iota_2)}
& {\,\,(\mathrm{K}_*(E_2),[1_{E_2}]) \,\,} \ar[r]_-{\mathrm{K}_*(\pi_2)}
& {\,\,(\mathrm{K}_*(A_2),[1_{A_2}]) \,\,} \ar[r]_-{}
& {\,\,0\,\,},}
$$ 
where $\mathrm{K}_*(\phi_0)$ is induced by $\phi_0$
and  $\varrho$ is the induced map obtained from $\mathrm{K}_*(\phi_0)$ and $\eta_*$.

Note that both $\mathrm{K}_*(\phi_0)$ and $\eta_*$ are scaled  order-preserving maps, \cite[Proposition 4]{LR} implies that
 $\varrho$ is also a scaled  order-preserving map.
  By Proposition \ref{simple k* to k total},
we can lift $\varrho$ to an isomorphism $\phi_1:\,A_1\to A_2$.
Now we have the following commutative diagram with exact rows:
$$
\xymatrixcolsep{2pc}
\xymatrix{
{\,\,0\,\,} \ar[r]^-{}
& {\,\,\mathrm{K}_*(B_1)\,\,} \ar[d]_-{{\rm id}} \ar[r]^-{\mathrm{K}_*(\iota_1)}
& {\,\,(\mathrm{K}_*(E_1),[1_{E_1}])\,\,} \ar[d]_-{\eta_*} \ar[r]^-{\mathrm{K}_*(\phi_1\circ \pi_1) }
& {\,\,(\mathrm{K}_*(A_2),[1_{A_2}])\,\,} \ar[d]_-{{\rm id}} \ar[r]^-{}
& {\,\,0\,\,\,} \\
{\,\,0\,\,} \ar[r]^-{}
& {\,\,\mathrm{K}_*(B_1)\,\,} \ar[r]_-{ \mathrm{K}_*( \iota_2\circ\phi_0)}
& {\,\,(\mathrm{K}_*(E_2),[1_{E_2}]) \,\,} \ar[r]_-{\mathrm{K}_*(\pi_2)}
& {\,\,(\mathrm{K}_*(A_2),[1_{A_2}]) \,\,} \ar[r]_-{}
& {\,\,0\,\,}.}
$$

Then by Theorem \ref{strong wei} (the naturality comes from \cite[Theorem 4.14]{GR}), the two extensions
$$
0\to B_1 \xrightarrow{\iota_1} E_1\xrightarrow{\phi_1\circ \pi_1} A_2\to 0
$$
and
$$
0\to B_1 \xrightarrow{\iota_2\circ \phi_0} E_2\xrightarrow{\pi_2} A_2\to 0
$$
are strongly unitarily equivalent,  then $E_1\cong E_2$.

\end{proof}

\section{The construction of Dadarlat-Eilers}
Before we raise the example of Theorem \ref{rho is nece}, we list a construction of Dadarlat-Eilers, which they applied to reveal the necessity of $\rho$. 
Their example has the commutativity on ${\rm K}_0(-;\mathbb{Z}_3)$, but we point out that the commutativity fails on ${\rm K}_0(-;\mathbb{Z}_9)$.


\subsection{Continuous field $C^*$-algebras}

\begin{definition}\label{construction F}\rm
We write $\alpha X$ for the one point (Alexandroff) compactification of a locally compact space $X$,
and denote the point at infinity by $\infty_X$. 
Suppose that $\{\phi_m\}: A \to B$ is a family of *-homomorphisms, we denote by
$$\mathcal{F}[(\phi_m)]\quad {\rm or}\quad \mathcal{F}[\phi_1, \phi_2, \phi_3, \phi_4,\cdots]$$
the C*-algebra
$$\{(a,(b_m))\in A\oplus \prod_{m=1}^\infty B:\, \|b_m-\phi_m(a)\|\to 0\}.
$$

Note that $\mathcal{F}[(\phi_m)]$ is an inductive limit of $C^*$-algebras of the form $A\oplus \bigoplus_1^{m}B$ using bonding maps of the form
$$
\chi_m(a,b_1,b_2,\cdots,b_m)=(a,b_1,b_2,\cdots,b_m,\phi_{m+1}(a)).
$$
\end{definition}
\begin{definition}\rm
  For real rank zero $\mathrm{AD}$ algebras with $n\cdot$tor\,${\rm K}_1(A)=0$,
$$
\mathbf{K}(A; n): {\rm K}_0(A) \xrightarrow{\rho_n^0} {\rm K}_0(A ; \mathbb{Z}_n) \xrightarrow{\beta_n^0} {\rm K}_1(A)
$$
is a complete invariant \cite{Ei}. For any *-homomorphism $\phi:A\to B$, $\phi$ induces $\underline{\rm K}(\phi):\underline{\rm K}(A)\to \underline{\rm K}(B)$, we identify $\mathbf{K}(\phi; n)$ as the triple
$$
({\rm K}_0(\phi),  {\rm K}_0(\phi ; \mathbb{Z}_n), {\rm K}_1(\phi)).
$$
We denote the resulting invariant $\mathbf{K}_{\langle\rho\rangle}(A; n)$, if one  deletes the $\rho_n^0$ map from $\mathbf{K}(A; n)$.
\end{definition}

\begin{notion}\rm \label{def AB}

Let us recall the construction of $D$ and $D'$ in \cite[Section 2]{DEb}. For convenience, we just concentrate on the algebras with $\mathbf{n=3}$.

Let $\{t_k\}$ be a sequence dense in $(0,1)$, define $\psi_k
: \widetilde{\mathbb{I}}_3\rightarrow M_4(\widetilde{\mathbb{I}}_3)$ by $\psi_k(f)={\rm diag}\{f,f(t_k),f(t_k),f(t_k)\}$, set
$$
\varphi_k=\psi_k\otimes {\rm id}: M_{4^{k+1}}(\widetilde{\mathbb{I}}_3)\rightarrow M_{4^{k+2}}(\widetilde{\mathbb{I}}_3).
$$
Let $A=\lim\limits_{\longrightarrow}(M_{4^{k+1}}(\widetilde{\mathbb{I}}_3),\varphi_k)$, then
$A$ is a unital simple AD algebra of stable rank one and real rank zero, and
$$(\mathrm{K}_0(A), \mathrm{K}_0^+(A), [1_A], \mathrm{K}_1(A))\cong (\mathbb{Z}[\frac{1}{4}],\mathbb{Z}[\frac{1}{4}]\cap \mathbb{R}_+,1,\mathbb{Z}_3),$$
where $\mathbb{Z}[\frac{1}{4}]=:\{\frac{j}{4^i}:\, i\in\mathbb{N},\,j\in \mathbb{Z}\}$ ($\mathbb{Z}[\frac{1}{4}]\cong \mathbb{Z}[\frac{1}{2}]=:\{\frac{j}{2^i}:\, i\in\mathbb{N},\,j\in \mathbb{Z}\}$)
and $\mathbb{R}_+=:[0,+\infty)$.

Denote $B$ to be the unital UHF algebra $M_{4^\infty}\otimes M_3$ with
$$(\mathrm{K}_0(B), \mathrm{K}_0^+(B), [1_B])\cong (\mathbb{Z}[\frac{1}{4}],\mathbb{Z}[\frac{1}{4}]\cap \mathbb{R}_+,3).$$
\end{notion}
\begin{notion}\label{DE con}\rm
 Choosing suitable generators, identify
$$\mathbf{K}(A; 3): \,\mathbb{Z}[ \frac{1}{4}]\xrightarrow{(\rho_A)_3^{0}}
\mathbb{Z}_3\oplus \mathbb{Z}_3
\xrightarrow{(\beta_A)_{3}^{0}} \mathbb{Z}_3,
$$$$
\mathbf{K}(B; 3):\, \mathbb{Z}[ \frac{1}{4}]
\xrightarrow{(\rho_B)_{3}^{0}}
 \mathbb{Z}_3
\xrightarrow{(\beta_B)_{3}^{0}} 0,$$
where
$$(\rho_A)_{3}^{0}\,(1) = ([1]_3, 0),\quad (\beta_A)_{3}^{0}\,(u, v) = v, \quad (\rho_B)_{3}^{0}\,(1) = [1]_3, \quad
(\beta_B)_{3}^{0}\,(u) = 0.$$
By a one-sided version of \cite[3.6]{DL2}, one can choose unital $*$-homomorphisms $\varphi,\varphi' :\, A \to B$
with
$$\mathbf{K}(\varphi; 3) = (3,
(0 ,1) , 0), \quad \mathbf{K}(\varphi' ; 3) = (3,
( 0 ,-1), 0).$$
Now let $D, D'$ be the C*-algebras (see \ref{construction F})
$$D = \mathcal{F}[\varphi, \varphi, \varphi, \varphi,\cdots],\quad D' = \mathcal{F}[\varphi,\varphi',\varphi,\varphi', \cdots].$$

Note that $D,D'$ have real rank zero and stable rank one. By \cite[Theorem 2.3]{DEb}, $ D\ncong D'$. Moreover,  by \cite[Theorem 9.1--9.3]{DG},
$$
 (\underline{\mathrm{K}}(D),\underline{\mathrm{K}}(D)_+)
 \ncong(\underline{\mathrm{K}}(D'),\underline{\mathrm{K}}(D')_+)
$$
as ordered $\Lambda$-modules.

We would also mention that for $k=9$, by choosing suitable generators, we have
$$
\mathrm{K}_0(A;\mathbb{Z}_9)\cong\mathbb{Z}_9\oplus \mathbb{Z}_3,\,\,\,\, \mathrm{K}_0(B;\mathbb{Z}_9)\cong\mathbb{Z}_9
$$ and
$$\mathbf{K}(A; 9): \,\mathbb{Z}[ \frac{1}{4}]\xrightarrow{(\rho_A)_{9}^{0}}
\mathbb{Z}_9\oplus \mathbb{Z}_3
\xrightarrow{(\beta_A)_{9}^{0}} \mathbb{Z}_3,
$$$$
\mathbf{K}(B; 9):\, \mathbb{Z}[ \frac{1}{4}]
\xrightarrow{(\rho_B)_{9}^{0}}
 \mathbb{Z}_9
\xrightarrow{(\beta_B)_{9}^{0}} 0,$$
where
$$(\rho_A)_{9}^{0}\,(1) = ([1]_9, 0),\,\,\,\, (\beta_A)_{9}^{0}\,(u, v) = v, \,\,\,\,(\rho_B)_{3}^{0}\,(1) = [1]_9, \,\,\,\,
(\beta_B)_{9}^{0}\,(u) = 0.$$

Note that the induced diagrams by
$\varphi,\varphi' :\, A \to B$ are
\begin{displaymath}
\xymatrixcolsep{4pc}
\xymatrix{
 \mathrm{K}_0(A)  \ar[r]^-{(\rho_A)_9^0}\ar[d]^-{3}&
 \mathrm{K}_0(A;\mathbb{Z}_9)  \ar[d]^-{(3,\iota)}\ar[r]^-{(\beta_A)_9^0} & \mathrm{K}_1(A)  \ar[d]^-{0} \\
 \mathrm{K}_0(B)  \ar[r]^-{(\rho_B)_9^0}&
 \mathrm{K}_0(B;\mathbb{Z}_9)  \ar[r]^-{(\beta_B)_9^0}& \mathrm{K}_1(B)
}
\end{displaymath}
and
\begin{displaymath}
\xymatrixcolsep{4pc}
\xymatrix{
 \mathrm{K}_0(A)  \ar[r]^-{(\rho_A)_9^0}\ar[d]^-{3}&
 \mathrm{K}_0(A;\mathbb{Z}_9)  \ar[d]^-{(3,-\iota)}\ar[r]^-{(\beta_A)_9^0} & \mathrm{K}_1(A)  \ar[d]^-{0} \\
 \mathrm{K}_0(B)  \ar[r]^-{(\rho_B)_9^0}&
 \mathrm{K}_0(B;\mathbb{Z}_9)  \ar[r]^-{(\beta_B)_9^0}& \mathrm{K}_1(B),
}
\end{displaymath}
respectively, where $\iota:\,\mathbb{Z}_3\to \mathbb{Z}_9$ with $\iota([1]_3)=[3]_9$.
So we have
$$\mathbf{K}(\varphi; 9) = (3,
(3 ,\iota) , 0)\neq (3,
(0 ,\iota) , 0)$$
and
$$\mathbf{K}(\varphi' ; 9) = (3,
( 3 ,-\iota), 0)\neq (3,
( 0 ,-\iota), 0).$$
\end{notion}
\subsection{Check the commutativity}

\begin{notion}\label{3 and 9}\rm



Denote $$l_k=\min \{\frac{k}{2^j}\in \mathbb{N}|\,j\in \mathbb{N}\},$$
i.e,
$$
(l_1,l_2,l_3,l_4,l_5,l_6,l_7,l_8,\cdots)=(1,1,3,1,5,3,7,1,\cdots).
$$
For each $k\geq1$,
via  the cofibre sequence, we have the following exact sequence
$$
\xymatrixcolsep{3pc}
\xymatrix{
{\mathbb{Z}[\frac{1}{4}]}  \ar[r]^-{\times k}
& {\mathbb{Z}[\frac{1}{4}]}  \ar[r]^-{}
& {\mathrm{K}_0(A; \mathbb{Z}_k)} \ar[d]_-{}
 \\
{\mathrm{K}_1(A; \mathbb{Z}_k)} \ar[u]_-{}
& {\mathbb{Z}_3} \ar[l]_-{}
& {\mathbb{Z}_3} \ar[l]_-{\times k},
}
$$
and by choosing suitable generators, we have
$$
{\mathrm{K}}_0(A;\mathbb{Z}_k)=
\begin{cases}
\mathbb{Z}_{l_k}\oplus \mathbb{Z}_3, & \mbox{if } 3\mid k \\
  \mathbb{Z}_{l_k}, & \mbox{if } 3\nmid k
\end{cases}
$$
and
$$
{\mathrm{K}}_1(A;\mathbb{Z}_k)=
\begin{cases}
\mathbb{Z}_3, & \mbox{if } 3\mid k \\
  0, & \mbox{if } 3\nmid k
\end{cases}.
$$
Similarly, for $k\geq1$, one may obtain
$$
{\mathrm{K}}_0(B;\mathbb{Z}_k)=
\mathbb{Z}_{l_k}, \quad {\mathrm{K}}_1(B;\mathbb{Z}_k)=0.
$$
Furthermore, we have
$$\mathrm{K}_0(\varphi)=\mathrm{K}_0(\varphi'):\, \mathbb{Z}[\frac{1}{4}]\xrightarrow{\times3} \mathbb{Z}[\frac{1}{4}]\quad
{\rm and}\quad
\mathrm{K}_1(\varphi)=\mathrm{K}_1(\varphi')=0;$$
If $3\mid k$, for $y=0,1,2$, we have
$$
\mathrm{K}_0(\varphi;\mathbb{Z}_k)([x]_{l_k},[y]_3)= 3\cdot [x]_{l_k}+y\cdot[\frac{{l_k}}{3}]_{l_k},
$$
and
$$\mathrm{K}_0(\varphi';\mathbb{Z}_k)([x]_{l_k},[y]_3)= 3\cdot [x]_{l_k}-y\cdot[\frac{{l_k}}{3}]_{l_k}.
$$
If $3\nmid k$, we have
$$
\mathrm{K}_0(\varphi;\mathbb{Z}_k)=\mathrm{K}_0(\varphi';\mathbb{Z}_k)=\times 3.
$$

 We point out that the following diagram is commutative for $k=3$, but not for $k=9$.
$$
\xymatrixcolsep{2pc}
\xymatrix{
{\,\,\mathrm{K}_0(A;\mathbb{Z}_k)\,\,} \ar[r]^-{\mathrm{K}_0(\varphi;\mathbb{Z}_k)}\ar[d]_-{id}
& {\,\,\mathrm{K}_0(B;\mathbb{Z}_k)\,\,} \ar[d]_-{ -id}
 \\
{\,\,\mathrm{K}_0(A;\mathbb{Z}_k) \,\,}\ar[r]_-{\mathrm{K}_0(\varphi';\mathbb{Z}_k)}
& {\,\,\mathrm{K}_0(B;\mathbb{Z}_k)\,\,}.}
$$

For $k=3$, the diagram
$$
\xymatrixcolsep{2pc}
\xymatrix{
{\,\,\mathrm{K}_0(A;\mathbb{Z}_3)\cong \mathbb{Z}_3\oplus \mathbb{Z}_3\,\,} \ar[r]^-{\mathrm{K}_0(\varphi;\mathbb{Z}_3)}\ar[d]_-{id}
& {\,\,\mathrm{K}_0(B;\mathbb{Z}_3)\cong \mathbb{Z}_3\,\,} \ar[d]_-{ -id}
  \\
{\,\,\mathrm{K}_0(A;\mathbb{Z}_3)\cong \mathbb{Z}_3\oplus \mathbb{Z}_3 \,\,}\ar[r]_-{\mathrm{K}_0(\varphi';\mathbb{Z}_3)}
& {\,\,\mathrm{K}_0(B;\mathbb{Z}_3)\cong \mathbb{Z}_3\,\,}}
$$
is commutative, where 
$$
\mathrm{K}_0(\varphi;\mathbb{Z}_3)([x]_3,[y]_3)=3\cdot [x]_3+[y]_3=[y]_3,
$$
$$
\mathrm{K}_0(\varphi';\mathbb{Z}_3)([x]_3,[y]_3)=3\cdot [x]_3-[y]_3=-[y]_3.
$$

For $k=9$, the diagram
$$
\xymatrixcolsep{2pc}
\xymatrix{
{\,\,\mathrm{K}_0(A;\mathbb{Z}_9)\cong \mathbb{Z}_9\oplus \mathbb{Z}_3\,\,} \ar[r]^-{\mathrm{K}_0(\varphi;\mathbb{Z}_9)}\ar[d]_-{id}
& {\,\,\mathrm{K}_0(B;\mathbb{Z}_9)\cong \mathbb{Z}_9\,\,} \ar[d]_-{ -id}
 \\
{\,\,\mathrm{K}_0(A;\mathbb{Z}_9)\cong \mathbb{Z}_9\oplus \mathbb{Z}_3 \,\,}\ar[r]_-{\mathrm{K}_0(\varphi';\mathbb{Z}_9)}
& {\,\,\mathrm{K}_0(B;\mathbb{Z}_9)\cong \mathbb{Z}_9\,\,},}
$$
is not commutative, where
$$
\mathrm{K}_0(\varphi;\mathbb{Z}_9)([x]_9,[y]_3)=3\cdot [x]_9+y\cdot[3]_9,$$
$$
\mathrm{K}_0(\varphi';\mathbb{Z}_9)([x]_9,[y]_3)=3\cdot [x]_9-y\cdot[3]_9,\quad y=0,1,2.
$$
In particular, take $([1]_9,[1]_3)\in \mathbb{Z}_9\oplus \mathbb{Z}_3$, we have
$$
-1\cdot\mathrm{K}_0(\varphi;\mathbb{Z}_9) ([1]_9,[1]_3)=
-[6]_9\neq [0]_9=\mathrm{K}_0(\varphi';\mathbb{Z}_9)\circ id ([1]_9,[1]_3).
$$

Moreover, any automorphism $\theta$ of $\mathrm{K}_0(B;\mathbb{Z}_9)(\cong
\mathbb{Z}_9)$ {\bf doesn't conjugate} $\mathrm{K}_0(\varphi;\mathbb{Z}_9)$
to $\mathrm{K}_0(\varphi';\mathbb{Z}_9)$, i.e.,
the diagram
$$
\xymatrixcolsep{2pc}
\xymatrix{
{\,\,\mathrm{K}_0(A;\mathbb{Z}_9)\,\,} \ar[r]^-{\mathrm{K}_0(\varphi;\mathbb{Z}_9)}\ar[d]_-{id}
& {\,\,\mathrm{K}_0(B;\mathbb{Z}_9)\,\,} \ar[d]_-{\theta}
 \\
{\,\,\mathrm{K}_0(A;\mathbb{Z}_9) \,\,}\ar[r]_-{\mathrm{K}_0(\varphi';\mathbb{Z}_9)}
& {\,\,\mathrm{K}_0(B;\mathbb{Z}_9)\,\,}}
$$
is not commutative, as $$\theta\circ\mathrm{K}_0(\varphi;\mathbb{Z}_9) ([1]_9,[1]_3)=\theta([6]_9)\neq [0]_9=\mathrm{K}_0(\varphi';\mathbb{Z}_9)\circ id ([1]_9,[1]_3).$$

One can check that the map induced by the conjugation ($-id_{\mathbb{Z}_9}$ on $\mathrm{K}_0(B;\mathbb{Z}_9)$) of Dadarlat-Eilers in \cite[proof of Theorem 2.3]{DEb}
will take the non-positive element
$$(1,(3,0,3,0,\cdots))\oplus(([1]_9,[1]_3),(-[6]_9,-[6]_9,-[6]_9,\cdots))\in\mathrm{K}_0(D)\oplus\mathrm{K}_0(D;\mathbb{Z}_9)$$
into the positive one
$$(1,(3,0,3,0,\cdots))\oplus(([1]_9,[1]_3),(-[6]_9,[0]_9,-[6]_9,[0]_9,\cdots))\in\mathrm{K}_0(D')\oplus\mathrm{K}_0(D';\mathbb{Z}_9).$$

In conclusion, 
the conjugation induces an ordered isomorphism
$$
\mathbf{K}_{\langle\rho\rangle}(D;3)\cong \mathbf{K}_{\langle\rho\rangle}(D';3),
$$
but it doesn't induce an ordered graded isomorphism such that
$$
\underline{\mathrm{K}}_{\langle\rho\rangle}(D)\cong \underline{\mathrm{K}}_{\langle\rho\rangle}(D').
$$
We will carefully adjust the construction to achieve the expected purpose.

\end{notion}
\section{The map $\rho$ is necessary for total K-theory}
In this section, we denote $A,B$ to be the stable algebras of original $A,B$ in \ref{def AB}.
That is, we take new $A,B$ as $A\otimes\mathcal{K},B\otimes\mathcal{K}$, respectively.

\subsection{A modified example}
\begin{notion}\rm  \label{def F1F2}
 Recall that for $k\geq 1$, we may identify
$$
{\mathrm{K}}_0(A;\mathbb{Z}_k)=
\begin{cases}
\mathbb{Z}_{l_k}\oplus \mathbb{Z}_3, & \mbox{if } 3\mid k \\
  \mathbb{Z}_{l_k}, & \mbox{if } 3\nmid k
\end{cases}.
$$
and
$$
{\mathrm{K}}_1(A;\mathbb{Z}_k)=
\begin{cases}
\mathbb{Z}_3, & \mbox{if } 3\mid k \\
  0, & \mbox{if } 3\nmid k
\end{cases}.
$$
Case 1: if $3\mid k$,
$$
\mathbf{K}(A; k): \,\mathbb{Z}[ \frac{1}{4}]\xrightarrow{(\rho_A)_{k}^{0}}
\mathbb{Z}_{l_k}\oplus \mathbb{Z}_3
\xrightarrow{(\beta_A)_{k}^{0}} \mathbb{Z}_3,
$$$$
\mathbf{K}(B; k):\, \mathbb{Z}[ \frac{1}{4}]
\xrightarrow{(\rho_B)_{k}^{0}}
 \mathbb{Z}_{l_k}
\xrightarrow{(\beta_B)_{k}^{0}} 0,$$
where
$$(\rho_A)_{k}^{0}\,(x) = ([x]_{l_k}, 0),\,\, (\beta_A)_{k}^{0}\,(u, v) = v, \,\, (\rho_B)_{k}^{0}\,(x) = [x]_{l_k}, \,\,
(\beta_B)_{k}^{0}\,(u) = 0.$$

Case 2: if $3\nmid k$,
$$\mathbf{K}(A; k): \,\mathbb{Z}[ \frac{1}{4}]\xrightarrow{(\rho_A)_{k}^{0}}
\mathbb{Z}_{l_k}
\xrightarrow{(\beta_A)_{k}^{0}} \mathbb{Z}_3,
$$$$
\mathbf{K}(B; k):\, \mathbb{Z}[ \frac{1}{4}]
\xrightarrow{(\rho_B)_{k}^{0}}
 \mathbb{Z}_{l_k}
\xrightarrow{(\beta_B)_{k}^{0}} 0,$$
where
$$(\rho_A)_{k}^{0}\,(x) = [x]_{l_k},\,\, (\beta_A)_{k}^{0}\,(u) = 0, \,\, (\rho_B)_{k}^{0}\,(x) = [x]_{l_k}, \,\,
(\beta_B)_{k}^{0}\,(u) = 0.$$


Now by a one-sided version of \cite[3.6]{DL2}, we may choose two collections of  *-homomorphisms $\{\omega_j\}_{j\geq 1},\{\omega_j'\}_{j\geq 1} :\, A \to B$
satisfying the following  for any {$j\geq k\geq 1$:}
$$
\mathbf{K}(\omega_j; k)=
\begin{cases}
 (l_{j!},(0 ,\iota_{j,k}) , 0), & \mbox{if } 3\mid k \\
 (l_{j!},
0  , 0), & \mbox{if } 3\nmid k
\end{cases},
$$
$$
\mathbf{K}(\omega_j'; k)=
\begin{cases}
 (l_{j!},(0 ,-\iota_{j,k}) , 0), & \mbox{if } 3\mid k \\
 (l_{j!},
0  , 0), & \mbox{if } 3\nmid k
\end{cases},
$$
where $\iota_{j,k}([x]_{3})=[x\cdot \frac{l_k}{3}]_{l_k}$.
(Note that $l_{2^i!}=l_{(2^i-1)!}$, then  we may assume $\omega_{2^i}=\omega_{2^i-1}$
and $\omega_{2^i}'=\omega_{2^i-1}'$.)

Now let $F_1, F_2$ be the C*-algebras (see \ref{construction F}),
$$F_1 = \mathcal{F}[\omega_1, \omega_1, \omega_2, \omega_2,\omega_3, \omega_3,\cdots],\quad F_2 = \mathcal{F}[\omega_1,\omega_1',\omega_2,\omega_2',\omega_3, \omega_3', \cdots].$$
(In \ref{3 and 9}, we find that the commutativity fails for $k=9$, as the two ``3'''s in
the second coordinates of $\mathbf{K}(\varphi; 9) = (3,
(3 ,\iota) , 0)$ and $ \mathbf{K}(\varphi' ; 9) = (3,
( 3 ,-\iota), 0)$, respectively, form an obstruction. So in the modification, once $k$ is fixed, for any $j\geq k$, the related commutativity considered in \ref{3 and 9} will always hold for $\omega_j$ and $\omega_j'$.)

\end{notion}


\begin{notion}\rm
From the inductive limit structures of $F_1$ and $F_2$, $F_1,\,F_2$ are stable and  one can find the following facts immediately.

(i) For $i=1,2$, $\mathrm{K}_0(F_i)$ consists of all the elements of the form
$$
(x_0,(x_1,x_2,x_3,\cdots))\in \mathbb{Z}[\frac{1}{4}] \oplus \prod_{n=1}^{\infty}\mathbb{Z}[\frac{1}{4}]
$$
with
$x_{2j-1}=x_{2j}=l_{j!}\cdot x_0$ for all large enough $j$ and
$\mathrm{K}_1(F_i)\cong \mathbb{Z}_3$.

(ii) If $3\mid k$,
$\mathrm{K}_0(F_1;\mathbb{Z}_k)$ consists of all the elements of the form
$$
(([x]_{l_k},[y]_3),(z_1,z_2,z_3,\cdots))\in \mathbb{Z}_{l_k}\oplus\mathbb{Z}_3 \oplus \prod_{n=1}^{\infty}\mathbb{Z}_{l_k}
$$
 with $z_j=y\cdot [\frac{l_k}{3}]_{l_k}$ ($y=0,1,2$) for all large enough $j$;
$\mathrm{K}_0(F_2;\mathbb{Z}_k)$ consists of all the elements of the form
$$
(([x]_{l_k},[y]_3),(z_1,z_2,z_3,\cdots))\in \mathbb{Z}_{l_k}\oplus\mathbb{Z}_3 \oplus \prod_{n=1}^{\infty}\mathbb{Z}_{l_k}
$$
 with $z_j=(-1)^{j+1}\cdot y\cdot [\frac{l_k}{3}]_{l_k}$ ($y=0,1,2$) for all large enough $j$.
We also have
$$
\mathrm{K}_1(F_1;\mathbb{Z}_k)\cong\mathrm{K}_1(F_2;\mathbb{Z}_k)\cong \mathbb{Z}_3.
$$

(iii) If $3\nmid k$,
$\mathrm{K}_0(F_i;\mathbb{Z}_k)$ ($i=1,2$) consists of all the elements of the form
$$
([x]_{l_k},(z_1,z_2,z_3,\cdots))\in \mathbb{Z}_{l_k} \oplus \prod_{n=1}^{\infty}\mathbb{Z}_{l_k}
$$
with $z_j=0$ for all large enough $j$
and
$\mathrm{K}_1(F_1;\mathbb{Z}_k)=\mathrm{K}_1(F_2;\mathbb{Z}_k)=0$.

\end{notion}


\begin{notion}\rm
For any $k\geq 1$, we confirm that following diagram is commutative for any $j\geq k$,
$$
\xymatrixcolsep{2pc}
\xymatrix{
{\,\,\mathrm{K}_0(A;\mathbb{Z}_k)\,\,} \ar[r]^-{\mathrm{K}_0(\omega_j;\mathbb{Z}_k)}\ar[d]_-{id}
& {\,\,\mathrm{K}_0(B;\mathbb{Z}_k)\,\,} \ar[d]_-{ -id}
 \\
{\,\,\mathrm{K}_0(A;\mathbb{Z}_k) \,\,}\ar[r]_-{\mathrm{K}_0(\omega_j';\mathbb{Z}_k)}
& {\,\,\mathrm{K}_0(B;\mathbb{Z}_k)\,\,},}
$$
which is $
\xymatrixcolsep{2pc}
\xymatrix{
{\,\,\mathbb{Z}_{l_k}\oplus \mathbb{Z}_3\,\,} \ar[r]^-{(0,\iota_{j,k})}\ar[d]_-{id}
& {\,\,\mathbb{Z}_{l_k}\,\,} \ar[d]_-{ -id}
 \\
{\,\,\mathbb{Z}_{l_k}\oplus \mathbb{Z}_3 \,\,}\ar[r]_-{(0,-\iota_{j,k})}
& {\,\,\mathbb{Z}_{l_k}\,\,}}
$ \,if $3\mid k$;
$
\xymatrixcolsep{2pc}
\xymatrix{
{\,\,\mathbb{Z}_{l_k}\,\,} \ar[r]^-{0}\ar[d]_-{id}
& {\,\,\mathbb{Z}_{l_k}\,\,} \ar[d]_-{ -id}
 \\
{\,\,\mathbb{Z}_{l_k} \,\,}\ar[r]_-{0}
& {\,\,\mathbb{Z}_{l_k}\,\,}}
$\,\,if $3\nmid k$.

Thus, for any $j\geq k\geq 1$,
the automorphism $-id$ on $\mathrm{K}_0(B;\mathbb{Z}_k)(\cong
\mathbb{Z}_{l_k})$  conjugates $\mathrm{K}_0(\omega_j ;\mathbb{Z}_k)$
to $\mathrm{K}_0(\omega_j';\mathbb{Z}_k)$. More concretely,
the family $G$ of quadruples of the form
$$
\xymatrixcolsep{2pc}
\xymatrix{
{\,\,\mathbf{K}_{\langle\rho\rangle}(B;\mathbb{Z}_k):\,\,}
& {\,\,0\,\,} \ar[r]^-{(\beta_B)_{k}^{1}}\ar[d]_-{id}
&{\,\,\mathrm{K}_0(B)\,\,}   \ar[d]^-{id}
& {\,\,\mathrm{K}_0(B;\mathbb{Z}_k)\,\,} \ar[d]^-{-id}\ar[r]^-{(\beta_B)_{k}^{1}}
& {\,\,0\,\,}\ar[d]_-{id}
 \\
{\,\,\mathbf{K}_{\langle\rho\rangle}(B;\mathbb{Z}_k): \,\,}
& {\,\,0\,\,} \ar[r]_-{(\beta_B)_{k}^{1}}
& {\,\,\mathrm{K}_0(B)\,\,}
& {\,\,\mathrm{K}_0(B;\mathbb{Z}_k)\,\,} \ar[r]_-{(\beta_B)_{k}^{1}}
& {\,\,0\,\,}.}
$$
consists of automorphisms of $\mathrm{K}_*(B; \mathbb{Z}\oplus\mathbb{Z}_k)$, which conjugate $\mathrm{K}_*(\omega_j;\mathbb{Z}\oplus\mathbb{Z}_k)$ to $\mathrm{K}_*(\omega_j';\mathbb{Z}\oplus\mathbb{Z}_k)$ for all $j\geq k$, are order preserving and respect the $\beta$ and $\kappa$ maps. (The following diagrams
$$
\xymatrixcolsep{2pc}
\xymatrix{
{\,\,\mathrm{K}_0(B;\mathbb{Z}_k)\,\,} \ar[r]^-{(\kappa_B)_{{kl,k}}^0}\ar[d]_-{-id}
&{\,\,\mathrm{K}_0(B;\mathbb{Z}_{kl})\,\,}   \ar[d]^-{-id}
& {\,\,\mathrm{K}_0(B;\mathbb{Z}_{kl})\,\,} \ar[d]^-{-id}\ar[r]^-{(\kappa_B)_{k,kl}^0}
& {\,\,\mathrm{K}_0(B;\mathbb{Z}_{k})\,\,}\ar[d]_-{-id}
 \\
{\,\,\mathrm{K}_0(B;\mathbb{Z}_{k})\,\,} \ar[r]_-{(\kappa_B)_{{kl,k}}^0}
& {\,\,\mathrm{K}_0(B;\mathbb{Z}_{kl})\,\,},
&{\,\,\mathrm{K}_0(B;\mathbb{Z}_{kl})\,\,} \ar[r]_-{(\kappa_B)_{{k,kl}}^0}
& {\,\,\mathrm{K}_0(B;\mathbb{Z}_{k})\,\,}}
$$
are commutative for all $k,l$.)
They induce by continuity of $\mathbf{K}_{\langle\rho\rangle}(-;\mathbb{Z}_k)$ an isomorphism
$$\mathbf{K}_{\langle\rho\rangle}(F_1;\mathbb{Z}_k)
\cong \mathbf{K}_{\langle\rho\rangle}(F_2;\mathbb{Z}_k),$$ since
$$
\xymatrixcolsep{3pc}
\xymatrix{
{\,\,\mathrm{K}_*(A\oplus \bigoplus\limits_{2j}B;\mathbb{Z}\oplus \mathbb{Z}_k)\,\,}
\ar[r]^-{(\chi_{2j+1}\circ\chi_{2j})_*}\ar[d]_-{id\oplus (id\oplus G)^j}
& {\,\,\mathrm{K}_*(A\oplus \bigoplus\limits_{2j+2}B;\mathbb{Z}\oplus \mathbb{Z}_k)\,\,}
\ar[d]^-{id\oplus (id\oplus G)^{j+1}}
 \\
{\,\,\mathrm{K}_*(A\oplus \bigoplus\limits_{2j}B;\mathbb{Z}\oplus \mathbb{Z}_k) \,\,}
\ar[r]_-{(\chi_{2j+1}'\circ\chi_{2j}')_*}
& {\,\,\mathrm{K}_*(A\oplus \bigoplus\limits_{2j+2}B;\mathbb{Z}\oplus \mathbb{Z}_k)\,\,}}
$$
commutes for all $j\geq k$, where the $\chi_{j}$ and $\chi_{j}'$ maps are bonding maps in the inductive limit description
of $F_1$ and $F_2$ as in \ref{construction F}.

In general, we obtained a  graded ordered group isomorphism $$\gamma:\, \underline{\mathrm{K}}(F_1)\to\underline{\mathrm{K}}(F_2)$$
whose restriction on $\mathrm{K}_*(F_1)$ is
$$\gamma_*=id:\,\mathrm{K}_*(F_1)\to \mathrm{K}_*(F_2);$$
if $3\mid k,$
$$
(([x]_{l_k},[y]_3),(z_1,z_2,z_3,z_4,\cdots))
\stackrel{ \gamma_k^0 }{\mapsto}
(([x]_{l_k},[y]_3),(z_1,-z_2,z_3,-z_4,\cdots))
$$
and
$\gamma_k^1=id_{\mathbb{Z}_3};$
if $3\nmid k,$
$$
([x]_{l_k},(z_1,z_2,z_3,z_4,\cdots))
\stackrel{ \gamma_k^0 }{\mapsto}
([x]_{l_k},(z_1,-z_2,z_3,-z_4,\cdots))$$
and $\gamma_k^1=0.$

Then $\gamma$  commutes with $\beta$ and $\kappa$ but not $\rho$.
(Note that for any $k\geq 3$ with $k$ odd,
$$
\xymatrixcolsep{2pc}
\xymatrix{
{\,\,\mathrm{K}_0(B)\,\,} \ar[r]^-{(\rho_B)_k^{0}}\ar[d]_-{id}
& {\,\,\mathrm{K}_0(B;\mathbb{Z}_k)\,\,} \ar[d]_-{ -id}
 \\
{\,\,\mathrm{K}_0(B) \,\,}\ar[r]_-{(\rho_B)_k^{0}}
& {\,\,\mathrm{K}_0(B;\mathbb{Z}_k)\,\,}.}
$$
is not a commutative diagram.)
\end{notion}

\subsection{The necessity of $\rho$ }

\begin{theorem}\label{thm mod DE}
There exist an ordered $\underline{\mathrm{K}}_{\langle\rho\rangle}$-isomorphism 
$$(\underline{\mathrm{K}}(F_1),\underline{\mathrm{K}}
(F_1)_+))_{\underline{\mathrm{K}}_{\langle\rho\rangle}}
\cong (\underline{\mathrm{K}}(F_2),\underline{\mathrm{K}}(F_2)_+)_{
\underline{\mathrm{K}}_{\langle\rho\rangle}}.
$$
while $F_1\ncong F_2$,  and hence,
$$(\underline{\mathrm{K}}(F_1),{\underline{\mathrm{K}}}(F_1)_+)_{\Lambda}\ncong (\underline{\mathrm{K}}(F_2),\underline{\mathrm{K}}(F_2)_+)_{\Lambda}.$$
\end{theorem}
\begin{proof}
We have obtain
$\gamma:\,\underline{\mathrm{K}}_{\langle\rho\rangle}(F_1)\cong \underline{\mathrm{K}}_{\langle\rho\rangle}(F_2)$
from the above argument.
Denote $\mathbb{N}_+=\mathbb{N}\backslash\{0\}$. As both $A,B$ are simple,  then both ${\rm Prim}(F_1)$, ${\rm Prim}(F_2)$
are homeomorphic to $\alpha\mathbb{N_+}$ and
there are split extensions
\vspace{-0.2cm}
\begin{center}
\begin{tikzpicture}
\node (a) at (0,0) {$0$};
\node (b) at (1.5,0) {$\bigoplus_{\mathbb{N}_+}B$};
\node (c) at (3,0) {$F_i$};
\node (d) at (4.5,0) {$A$};
\node (e) at (6.5,0) {$0,\quad i=1,2,$};
\path[->,font=\scriptsize]
(a) edge (b)
(b) edge (c)
([yshift= 1pt]c.east) edge node[above] {$\pi_i$} ([yshift= 1pt]d.west)
([yshift= -1pt]d.west) edge node[below] {$\iota_i$} ([yshift= -1pt]c.east)
(d) edge (e);
\end{tikzpicture}
\end{center}
\vspace{-0.5cm}
where
$$\pi_i(a,\prod_{\mathbb{N}_+}b_m)=a,\quad (a,\prod_{\mathbb{N}_+}b_m)\in F_i,\quad  i=1,2,$$
$$\iota_1(a)=(a,\omega_1(a),\omega_1(a),\omega_2(a),\omega_2(a),\cdots)$$
and  $$\iota_2(a)=(a,\omega_1(a),\omega_1'(a),\omega_2(a),\omega_2'(a),\cdots).$$

{\bf Step 1:} We begin by
 assuming that $\Xi:\,F_2\cong F_1$. Then we may obtain an induced homeomorphism   $g:{\rm Prim}(F_2)\cong{\rm Prim}(F_1)$.
Since $\infty_{\mathbb{N}_+}$ is the only non-isolated point of $\alpha\mathbb{N}_+$, $g(\infty_{\mathbb{N}_+}) = \infty_{\mathbb{N}_+}$, and so $g$ restricts
to a permutation $\sigma$ of $\mathbb{N}_+$. More concretely, we get the following commutative diagram
$$
\xymatrixcolsep{2pc}
\xymatrix{
{\,\,0\,\,} \ar[r]^-{}
& {\,\,\bigoplus_{\mathbb{N}_+}B\,\,} \ar[r]_-{}\ar[d]_-{\Xi_0}
&{\,\,F_2\,\,}  \ar[r]^-{\pi_2} \ar[d]^-{\Xi}
& {\,\,A\,\,} \ar[r]^-{}\ar[d]_-{\Xi_1}
& {\,\,0\,\,}
 \\
{\,\,0 \,\,}\ar[r]_-{}
& {\,\,\bigoplus_{\mathbb{N}_+}B\,\,} \ar[r]_-{}
& {\,\,F_1\,\,} \ar[r]_-{\pi_1}
& {\,\,A \,\,} \ar[r]_-{}
& {\,\,0\,\,},}
$$
where $\Xi_0=\mathop{\prod}\limits_{m\geq 1}\xi_m$ ($\xi_m:\, B^{m^{\rm th}}\to B^{\sigma(m)^{\rm th}}$ is the obvious automorphism of $B$) is the restriction of $\Xi$ and $\Xi_1$ is the induced map of $\Xi$.

Moreover, for each $m\geq 1$,
denote
$J_m$ and $J_{m}'$ to be the standard ideals of $F_1,F_2$ vanishing at the ${m^{\rm th}}$ $B$ copies, respectively.
Then there are split extensions
\vspace{-0.5cm}
\begin{center}
\begin{tikzpicture}
\node (a) at (0,0) {$0$};
\node (b) at (1.5,0) {$J_m$};
\node (c) at (3,0) {$F_1$};
\node (d) at (4.5,0) {$B$};
\node (e) at (6,0) {$0$};
\path[->,font=\scriptsize]
(a) edge (b)
(b) edge (c)
([yshift= 1pt]c.east) edge node[above] {$\it{\zeta_m}$} ([yshift= 1pt]d.west)
([yshift= -1pt]d.west) edge node[below] {$\it{\epsilon_m}$} ([yshift= -1pt]c.east)
(d) edge (e);
\end{tikzpicture}
\end{center}
\vspace{-0.5cm}
and
\vspace{-0.5cm}
\begin{center}
\begin{tikzpicture}
\node (a) at (0,0) {$0$};
\node (b) at (1.5,0) {$J_m'$};
\node (c) at (3,0) {$F_2$};
\node (d) at (4.5,0) {$B$};
\node (e) at (6,0) {$0$,};
\path[->,font=\scriptsize]
(a) edge (b)
(b) edge (c)
([yshift= 1pt]c.east) edge node[above] {$\it{\zeta_m'}$} ([yshift= 1pt]d.west)
([yshift= -1pt]d.west) edge node[below] {$\it{\epsilon_m'}$} ([yshift= -1pt]c.east)
(d) edge (e);
\end{tikzpicture}
\end{center}
\vspace{-0.5cm}
where $\zeta_m,\zeta_m'$ are the quotient maps and $\epsilon_m,\epsilon_m'$ are inclusions that
$$
\epsilon_m(b_m)=\epsilon_m'(b_m)=(0,(\underbrace{0,0,\cdots,0}_{m-1},b_m,0,\cdots))\in F_1\cap F_2,\quad
\forall\, b_m\in B.
$$
Then we have the commutative diagram
$$
\xymatrixcolsep{2pc}
\xymatrix{
{\,\,0\,\,} \ar[r]^-{}
& {\,\,J_m'\,\,} \ar[r]_-{}\ar[d]_-{\Xi_0^m}
&{\,\,F_2\,\,}  \ar[r]^-{\zeta_m'} \ar[d]_-{\Xi}
& {\,\,B\,\,} \ar[r]^-{}\ar[d]_-{\Xi_1^m}
& {\,\,0\,\,}
 \\
{\,\,0 \,\,}\ar[r]_-{}
& {\,\,J_{\sigma(m)}\,\,} \ar[r]_-{}
& {\,\,F_1\,\,} \ar[r]_-{\zeta_{\sigma(m)}}
& {\,\,B \,\,} \ar[r]_-{}
& {\,\,0\,\,},}
$$
where $\Xi_0^m$ is the restriction map and $\Xi_1^m$ is the induced map of $\Xi$.

For each $m\geq 1$, we also have the following exact sequences:

$(*)$ $0\to {\mathrm{K}_0(J_m')} \to {\mathrm{K}_0(F_2)}\to {\mathrm{K}_0(B)}\to 0;$

$(**)$ $0\to{\mathrm{K}_0(J_m; \mathbb{Z}_3)}  \to {\mathrm{K}_0(F_1; \mathbb{Z}_3)}  \to {\mathrm{K}_0(B; \mathbb{Z}_3)} \to 0.$

{\bf Step 2:} For $j\geq 1$, consider the following diagram (which is not necessarily commutative):

\vspace{-0.5cm}
$$
\begin{tikzpicture}[->,>=stealth,auto,node distance=12em, thick]

\node (X) at (0,0){$A$};
\node (Y) at (4,0){$F_2$};
\node (T) at (2,2.5){$B$};
\node (Z) at (7,0){$F_1$};
\node (Q) at(9,2.5){$B$};
\node (P)at (11,0){$A,$};

\draw[->] (X.26) -- (Y.157) node[midway, above] {$\iota_2$};
\draw[->] (Y.196) -- (X.-20) node[midway, below] {$\pi_2$};
\draw[->] (X) -- (T) node[midway, above] {$\omega_j'$\,\,\,\,\,\,};
\draw[->] (Y.75) -- (T.-50) node[midway,above] {\,\,\,\,$\zeta_{2j}'$};
\draw[->] (T.-90) -- (Y.120) node[midway, below] {$\epsilon_{2j}'$\,\,\,\,};
\draw[->] (Y.23) -- (Z.157) node[midway, above] {$\Xi$};
\draw[->] (Z.196) -- (Y.-15) node[midway, below] {$\Xi^{-1}$};
\draw[->] (Z.25) -- (P.153) node[midway, above] {\,\,\,\,\,\,\,\,\,\,\,\,$\pi_1$};
\draw[->] (P.196) -- (Z.-15) node[midway, below] {\,\,\,\,\,\,\,\,\,\,\,\,$\iota_1$};
\draw[->] (Z.70) -- (Q.230) node[midway, below] {\,\,\,\,\,\,\,\,\,$\zeta_{\sigma{(2j)}}$};
\draw[->] (Q.200) -- (Z.110) node[midway, above] {$\epsilon_{\sigma{(2j)}}$\,\,\,\,\,\,\,\,\,\,};
\draw[->] (P) -- (Q) node[midway,above] {\,\,\,\,\,\,\,\,\,\,\,\,\,\,\,\,$\omega_{\lceil\frac{\sigma(2j)}{2}\rceil}$};
\end{tikzpicture}
\vspace{-0.3cm}
$$
where $\lceil\frac{\sigma(2j)}{2}\rceil=\inf\{r\in \mathbb{N}\mid r\geq \frac{\sigma(2j)}{2}\}.$

We have
$$
\omega_j'=\zeta_{2j}'\circ\iota_2,\,\,\,\,
\omega_{\lceil\frac{\sigma(2j)}{2}\rceil}=\zeta_{\sigma(2j)}\circ\iota_1,
\,\,\,\,
\zeta_{2j}'\circ\epsilon_{2j}'=
\zeta_{\sigma(2j)}\circ\epsilon_{\sigma(2j)}=id_B\in{\rm Aut}(B)$$
$$
\pi_1\circ\iota_1=
\pi_2\circ\iota_2=id_A\in{\rm Aut}(A),
\,\,\,\,
\Xi_1=\pi_1\circ\Xi\circ\iota_{2},
\,\,\,\,
\Xi_1^{2j}=\zeta_{\sigma(2j)}\circ\Xi\circ\epsilon_{2j}'.
$$

For $1\in \mathrm{K}_0(A)\cong \mathbb{Z}[\frac{1}{4}]$,
$$\mathrm{K}_0(\omega_j')=\mathrm{K}_0(\zeta_{2j}')\circ\mathrm{K}_0(\iota_2):\mathrm{K}_0(A)\to \mathrm{K}_0(B).$$
Multiply $\mathrm{K}_0(\epsilon_{2j}')$ on the left side and by $(*)$ in {\bf Step 1},
we have
$$\mathrm{K}_0(\epsilon_{2j}'\circ\omega_j')(1)-
\mathrm{K}_0(\iota_2)(1)\in \mathrm{K}_0(J_{2j}')\subset\mathrm{K}_0(F_2).
$$
Thus,
$$\mathrm{K}_0(\Xi\circ\epsilon_{2j}'\circ\omega_j')(1)-
\mathrm{K}_0(\Xi\circ\iota_2)(1)\in \mathrm{K}_0(J_{\sigma(2j)})\subset\mathrm{K}_0(F_1) ,
$$
which means,
$$\mathrm{K}_0(\zeta_{\sigma(2j)}\circ\Xi\circ\epsilon_{2j}'\circ\omega_j')(1)=
\mathrm{K}_0(\zeta_{\sigma(2j)}\circ\Xi\circ\iota_2)(1)\in \mathrm{K}_0(B).
$$
Note that
$\mathrm{K}_0(\Xi_1)=\mathrm{K}_0(\pi_1\circ\Xi\circ\iota_2)\in {\rm Aut}(\mathrm{K}_0(A))$,
there exists an $n\in \mathbb{Z}$ such that
$
\mathrm{K}_0(\Xi_1)(1)=2^n.$
Then $\mathrm{K}_0(\Xi\circ\iota_2)(1)$ has the form
$$(2^n,(x_m)_{m\geq1})\in \mathrm{K}_0(F_1)\subset \mathrm{K}_0(A)\oplus \prod_{m=1}^\infty\mathrm{K}_0(B),$$
where $x_{2j-1}=x_{2j}=2^n\cdot l_{j!}$ for all large enough $j$. As $\mathrm{K}_0(\iota_1)(2^n)=(2^n,(y_m)_{m\geq1})\in \mathrm{K}_0(F_1)$,
where $y_{2j-1}=y_{2j}=2^n\cdot l_{j!}$ for all $j\geq 1$. Then there exists a large enough $j_0\geq 3$ such that for all $j\geq j_0$, we have
$$
\mathrm{K}_0(\zeta_{\sigma(2j)}\circ\Xi\circ\epsilon_{2j}')\circ\mathrm{K}_0(\omega_j')=
\mathrm{K}_0(\zeta_{\sigma(2j)}\circ \iota_1)\circ\mathrm{K}_0(\pi_1\circ\Xi\circ\iota_2),
$$
and hence, the following diagram is commutative.
$$
\xymatrixcolsep{3pc}
\xymatrix{
{\mathrm{K}_0(A)} \cong {\mathbb{Z}[\frac{1}{4}]} \ar[r]^-{\mathrm{K}_0(\Xi_1)} \ar[d]^-{\mathrm{K}_0(\omega_j')=l_{j!}}
& {\mathrm{K}_0(A)} \cong {\mathbb{Z}[\frac{1}{4}]} \ar[d]^-{\mathrm{K}_0(\omega_{\lceil\frac{\sigma(2j)}{2}\rceil})}
 \\
{\mathrm{K}_0(B)}\cong {\mathbb{Z}[\frac{1}{4}]} \ar[r]_-{\mathrm{K}_0(\Xi_1^{2j})}
& {\mathrm{K}_0(B)}\cong {\mathbb{Z}[\frac{1}{4}]}.
}
$$

Note that $\mathrm{K}_0(\Xi_1^{2j})\in {\rm Aut} ({\mathbb{Z}[\frac{1}{4}]})$
and $2\nmid l_{j!}$, $2\nmid l_{\lceil\frac{\sigma(2j)}{2}\rceil!}$, immediately,
we obtain
$$\mathrm{K}_0(\Xi_1^{2j})(1)=2^n\quad {\rm and}\quad \mathrm{K}_0(\omega_{\lceil\frac{\sigma(2j)}{2}\rceil})=l_{j!}.$$
That is,
 $\omega_{\lceil\frac{\sigma(2j)}{2}\rceil}=\omega_{j}$ for all $j\geq j_0$.
(We point out that $\sigma(2j)$ is not necessarily $2j-1$ or $2j$---for example, $\sigma(2^{i+1})$ can even be $2^{i+1}-3$ or $2^{i+1}-2$ for some large enough integer $i$, as we  have $\omega_{2^{i}}=\omega_{2^{i}-1}$.)

Similarly, for all large enough $j$, we have the following diagram
$$
\xymatrixcolsep{3pc}
\xymatrix{
{\mathrm{K}_0(A)}\cong {\mathbb{Z}[\frac{1}{4}]}  \ar[r]^-{\mathrm{K}_0(\Xi_1)} \ar[d]^-{\mathrm{K}_0(\omega_j)}
& {\mathrm{K}_0(A)\cong {\mathbb{Z}[\frac{1}{4}]}}  \ar[d]^-{\mathrm{K}_0(\omega_{\lceil\frac{\sigma(2j-1)}{2}\rceil})}
 \\
{\mathrm{K}_0(B)\cong {\mathbb{Z}[\frac{1}{4}]}} \ar[r]_-{\mathrm{K}_0(\Xi_1^{2j-1})}
& {\mathrm{K}_0(B)\cong {\mathbb{Z}[\frac{1}{4}]}}
}
$$
is commutative and
$\omega_{\lceil\frac{\sigma(2j-1)}{2}\rceil}=\omega_{j}$.


{\bf Step 3:} 
Consider the following commutative diagram:
\begin{displaymath}
\xymatrixcolsep{4pc}
\xymatrix{
 \mathrm{K}_0(A)  \ar[r]^-{(\rho_A)_3^0}\ar[d]^-{\mathrm{K}_0(\Xi_1)}&
 \mathrm{K}_0(A;\mathbb{Z}_3)  \ar[d]^-{\mathrm{K}_0(\Xi_1;\mathbb{Z}_3)}\ar[r]^-{(\beta_A)_3^0} & \mathrm{K}_1(A)  \ar[d]^-{\mathrm{K}_1(\Xi_1)} \\
 \mathrm{K}_0(A)  \ar[r]^-{(\rho_A)_3^0}&
 \mathrm{K}_0(A;\mathbb{Z}_3)  \ar[r]^-{(\beta_A)_3^0}& \mathrm{K}_1(A).
}
\end{displaymath}

As we have shown in {\bf Step 2} that $\mathrm{K}_0(\Xi_1)(1)=2^n$, from the above commutativity, we obtain
$$
\mathrm{K}_0(\Xi_1;\mathbb{Z}_3)([1]_3,[0]_3)=
\begin{cases}
  ([1]_3,[0]_3), & \mbox{if } n\,{\rm is \,even} \\
   ([2]_3,[0]_3), & \mbox{if } n\,{\rm is \,odd}
\end{cases}.
$$
Since $\mathrm{K}_1(\Xi_1)$ can be $id_{\mathbb{Z}_3}$ or $-id_{\mathbb{Z}_3}$, then we have the following cases:

{\bf Case 1:} $\mathrm{K}_1(\Xi_1)=id_{\mathbb{Z}_3}$  and $n$ is even.

{\bf Case 2:} $\mathrm{K}_1(\Xi_1)=id_{\mathbb{Z}_3}$ and $n$ is odd.

{\bf Case 3:} $\mathrm{K}_1(\Xi_1)=-id_{\mathbb{Z}_3}$ and $n$ is even.

{\bf Case 4:} $\mathrm{K}_1(\Xi_1)=-id_{\mathbb{Z}_3}$ and $n$ is odd.

We only need to deal with {\bf Case 1}, the other cases are very similar. Since $$\mathrm{K}_1(\Xi_1)\circ(\beta_A)_3^0=(\beta_A)_3^0\circ \mathrm{K}_0(\Xi_1;\mathbb{Z}_3),$$
then  $\mathrm{K}_0(\Xi_1;\mathbb{Z}_3)([0]_3,[1]_3)=([a]_3,[1]_3)$, for some $a\in\{0,1,2\}$.

Our goal is to check the commutativity of the following diagram
$$
\xymatrixcolsep{3pc}
\xymatrix{
{\mathrm{K}_0(A;\mathbb{Z}_3)}  \ar[r]^-{\mathrm{K}_0(\Xi_1;\mathbb{Z}_3)} \ar[d]^-{\mathrm{K}_0(\omega_j';\mathbb{Z}_3)}
& {\mathrm{K}_0(A;\mathbb{Z}_3)}  \ar[d]^-{\mathrm{K}_0(\omega_j;\mathbb{Z}_3)}
 \\
{\mathrm{K}_0(B;\mathbb{Z}_3)} \ar[r]_-{\mathrm{K}_0(\Xi_1^{2j};\mathbb{Z}_3)}
& {\mathrm{K}_0(B;\mathbb{Z}_3)}
}
$$
for all large enough integer $j$. Since $\mathrm{K}_0(A;\mathbb{Z}_3)\cong \mathbb{Z}_3\oplus \mathbb{Z}_3$, we only need to check the generators $([1]_3,[0]_3)$ and $([0]_3,[1]_3)$. The procedure is very similar to {\bf
 Step 1}.

For the generator $([0]_3,[1]_3)$, it is easily seen that for all large enough $j$,
$$\mathrm{K}_0(\omega_j;\mathbb{Z}_3)
\circ\mathrm{K}_0(\Xi_1;\mathbb{Z}_3)([1]_3,[0]_3)
=\mathrm{K}_0(\omega_j;\mathbb{Z}_3)([1]_3,[0]_3)$$$$
=[0]_3=\mathrm{K}_0(\Xi_1^{2j};\mathbb{Z}_3)([0]_3,[0]_3)
=\mathrm{K}_0(\Xi_1^{2j};\mathbb{Z}_3)\circ
\mathrm{K}_0(\omega_j';\mathbb{Z}_3)
([1]_3,[0]_3).$$

For the generator $([0]_3,[1]_3)$ and large enough $j$, consider the following diagram (which is not necessarily commutative):
$$
\begin{tikzpicture}[->,>=stealth,auto,node distance=12em, thick]

\node (X) at (0,0){$A$};
\node (Y) at (4,0){$F_2$};
\node (T) at (2,2.5){$B$};
\node (Z) at (7,0){$F_1$};
\node (Q) at(9,2.5){$B$};
\node (P)at (11,0){$A.$};

\draw[->] (X.26) -- (Y.157) node[midway, above] {$\iota_2$};
\draw[->] (Y.196) -- (X.-20) node[midway, below] {$\pi_2$};
\draw[->] (X) -- (T) node[midway, above] {$\omega_j'$\,\,\,\,\,\,};
\draw[->] (Y.75) -- (T.-50) node[midway,above] {\,\,\,\,$\zeta_{2j}'$};
\draw[->] (T.-90) -- (Y.120) node[midway, below] {$\epsilon_{2j}'$\,\,\,\,};
\draw[->] (Y.23) -- (Z.157) node[midway, above] {$\Xi$};
\draw[->] (Z.196) -- (Y.-15) node[midway, below] {$\Xi^{-1}$};
\draw[->] (Z.25) -- (P.155) node[midway, above] {\,\,\,\,\,\,\,\,\,\,\,\,$\pi_1$};
\draw[->] (P.196) -- (Z.-15) node[midway, below] {\,\,\,\,\,\,\,\,\,\,\,\,$\iota_1$};
\draw[->] (Z.70) -- (Q.230) node[midway, below] {\,\,\,\,\,\,\,\,\,$\zeta_{\sigma{(2j)}}$};
\draw[->] (Q.200) -- (Z.110) node[midway, above] {$\epsilon_{\sigma{(2j)}}$\,\,\,\,\,\,\,\,\,\,};
\draw[->] (P) -- (Q) node[midway,above] {\,\,\,\,\,\,\,\,\,\,\,\,\,\,\,\,\,\,\,\,\,\,\,\,\,\,\,\,$\omega_{\lceil\frac{\sigma(2j)}{2}\rceil}=\omega_j$};
\end{tikzpicture}$$
Recall that $$\omega_j'=\zeta_{2j}'\circ\iota_2,\,\,\,\,
\omega_{\lceil\frac{\sigma(2j)}{2}\rceil}=\zeta_{\sigma(2j)}\circ\iota_1,
\,\,\,\,
\zeta_{2j}\circ\epsilon_{2j}=
\zeta_{\sigma(2j)}'\circ\epsilon_{\sigma(2j)}'=id_B,$$
$$
\pi_1\circ\iota_1=
\pi_2\circ\iota_2=id_A,
\,\,\,\,
\Xi_1=\pi_1\circ\Xi\circ\iota_{2},
\,\,\,\,
\Xi_1^{2j}=\zeta_{\sigma(2j)}\circ\Xi\circ\epsilon_{2j}'.
$$

Then we have
$$\mathrm{K}_0(\omega_j';\mathbb{Z}_3)=
\mathrm{K}_0(\zeta_{2j}';\mathbb{Z}_3)\circ\mathrm{K}_0(\iota_2;\mathbb{Z}_3).$$
By the exactness $(**)$ in {\bf Step 1}, we have
$$\mathrm{K}_0(\epsilon_{2j}'\circ\omega_j';\mathbb{Z}_3)([0]_3,[1]_3)-
\mathrm{K}_0(\iota_2;\mathbb{Z}_3)([0]_3,[1]_3)\in \mathrm{K}_0(J_{2j}';\mathbb{Z}_3).
$$
Thus,
$$\mathrm{K}_0(\Xi\circ\epsilon_{2j}'\circ\omega_j';\mathbb{Z}_3)([0]_3,[1]_3)-
\mathrm{K}_0(\Xi\circ\iota_2;\mathbb{Z}_3)([0]_3,[1]_3)\in \mathrm{K}_0(J_{\sigma(2j)};\mathbb{Z}_3),
$$
which means,
$$\mathrm{K}_0(\zeta_{\sigma(2j)}\circ\Xi\circ\epsilon_{2j}'\circ\omega_j';\mathbb{Z}_3)([0]_3,[1]_3)=
\mathrm{K}_0(\zeta_{\sigma(2j)}\circ\Xi\circ\iota_2;\mathbb{Z}_3)([0]_3,[1]_3)\in \mathrm{K}_0(B;\mathbb{Z}_3),
$$
i.e.,
$$\mathrm{K}_0(\Xi_1^{2j}\circ\omega_j';\mathbb{Z}_3)([0]_3,[1]_3)=
\mathrm{K}_0(\zeta_{\sigma(2j)}\circ\Xi\circ\iota_2;\mathbb{Z}_3)([0]_3,[1]_3).
$$

Note that $$
\mathrm{K}_0(\iota_1;\mathbb{Z}_3)([a]_{3},[1]_3)=
(([a]_{3},[1]_3),([1]_3,[1]_3,[1]_3,\cdots))\in \mathbb{Z}_{3}\oplus\mathbb{Z}_3 \oplus \prod_{n=1}^{\infty}\mathbb{Z}_{3}.
$$
Since
we have the split sequence
\vspace{-0.2cm}
\begin{center}
\begin{tikzpicture}
\node (a) at (0,0) {$0$};
\node (b) at (1.5,0) {$\bigoplus_{\mathbb{N}_+}B$};
\node (c) at (3,0) {$F_1$};
\node (d) at (4.5,0) {$A$};
\node (e) at (6.0,0) {$0,$};
\path[->,font=\scriptsize]
(a) edge (b)
(b) edge (c)
([yshift= 1pt]c.east) edge node[above] {$\pi_1$} ([yshift= 1pt]d.west)
([yshift= -1pt]d.west) edge node[below] {$\iota_1$} ([yshift= -1pt]c.east)
(d) edge (e);
\end{tikzpicture}
\end{center}
\vspace{-0.3cm}
$\mathrm{K}_0(\Xi_1\circ \iota_2;\mathbb{Z}_3)([0]_3,[1]_3)$
has the form
$$
(([a]_{3},[1]_3),(z_1,z_2,z_3,\cdots))\in \mathbb{Z}_{3}\oplus\mathbb{Z}_3 \oplus \prod_{n=1}^{\infty}\mathbb{Z}_{3}
$$
 with $z_j= [1]_3$ for all large enough $j$.

 Then there exists a $j_1\geq j_0$ such that for all $j\geq j_1$, we have
\begin{align*}
  &\mathrm{K}_0(\Xi_1^{2j}\circ\omega_j';\mathbb{Z}_3)([0]_3,[1]_3)\nonumber\\
  & =\mathrm{K}_0(\zeta_{\sigma(2j)}\circ \iota_1;\mathbb{Z}_3)\circ\mathrm{K}_0(\pi_1\circ\Xi\circ\iota_2;\mathbb{Z}_3)([0]_3,[1]_3)\nonumber\\
   & = \mathrm{K}_0(\omega_j;\mathbb{Z}_3)\circ\mathrm{K}_0(\Xi_1;\mathbb{Z}_3)([0]_3,[1]_3) \\
   & =[1]_3. \qquad (***)
\end{align*}
Since $\Xi_1^{2j}$ is an isomorphism, via the coefficient maps, we have
the following commutative diagram
$$
\xymatrixcolsep{3pc}
\xymatrix{
{\mathrm{K}_0(B)} \ar[d]^-{\mathrm{K}_0(\Xi_1^{2j})} \ar[r]^-{\rho_3^0}
&{\mathrm{K}_0(B;\mathbb{Z}_3)} \ar[d]^-{\mathrm{K}_0(\Xi_1^{2j};\mathbb{Z}_3)}
 \\
{\mathrm{K}_0(B)} \ar[r]^-{\rho_3^0}
& {\mathrm{K}_0(B;\mathbb{Z}_3)}.
}
$$
As $n$ is even, $\rho_3^0\circ\mathrm{K}_0(\Xi_1^{2j})(1)=[1]_3$, and hence, $\mathrm{K}_0(\Xi_1^{2j};\mathbb{Z}_3)=id_{\mathbb{Z}_3}$.

Then for all  $j\geq j_1$, we have
$$
\mathrm{K}_0(\Xi_1^{2j};\mathbb{Z}_3)\circ
\mathrm{K}_0(\omega_j';\mathbb{Z}_3)
([0]_3,[1]_3)=
\mathrm{K}_0(\Xi_1^{2j};\mathbb{Z}_3)([2]_3)
=[2]_3,
$$
which contradicts with $(***)$.


For Cases 2, 3, 4, the similar arguments lead to similar contradictions.
The different part is that the Case 2 and 3 should be dealt with
the diagram
$$
\begin{tikzpicture}[->,>=stealth,auto,node distance=12em, thick]

\node (X) at (0,0){$A$};
\node (Y) at (4,0){$F_2$};
\node (T) at (2,2.5){$B$};
\node (Z) at (7,0){$F_1$};
\node (Q) at(9,2.5){$B$};
\node (P)at (11,0){$A.$};

\draw[->] (X.26) -- (Y.157) node[midway, above] {$\iota_2$};
\draw[->] (Y.196) -- (X.-20) node[midway, below] {$\pi_2$};
\draw[->] (X) -- (T) node[midway, above] {$\omega_j$\,\,\,\,\,\,};
\draw[->] (Y.75) -- (T.-50) node[midway,above] {\,\,\,\,\,\,\,$\zeta_{2j-1}'$};
\draw[->] (T.-90) -- (Y.120) node[midway, below] {$\epsilon_{2j-1}'$\,\,\,\,};
\draw[->] (Y.23) -- (Z.157) node[midway, above] {$\Xi$};
\draw[->] (Z.196) -- (Y.-15) node[midway, below] {$\Xi^{-1}$};
\draw[->] (Z.25) -- (P.155) node[midway, above] {\,\,\,\,\,\,\,\,\,\,\,\,$\pi_1$};
\draw[->] (P.196) -- (Z.-15) node[midway, below] {\,\,\,\,\,\,\,\,\,\,\,\,$\iota_1$};
\draw[->] (Z.70) -- (Q.230) node[midway, below] {\,\,\,\,\,\,\,\,\,\,\,\,\,\,\,\,\,\,\,$\zeta_{\sigma{(2j-1)}}$};
\draw[->] (Q.200) -- (Z.110) node[midway, above] {$\epsilon_{\sigma{(2j-1)}}$\,\,\,\,\,\,\,\,\,\,\,\,\,\,\,};
\draw[->] (P) -- (Q) node[midway,above] {\,\,\,\,\,\,\,\,\,\,\,\,\,\,\,\,\,\,\,\,\,\,\,\,\,\,\,\,\,\,\,\,$\omega_{\lceil\frac{\sigma(2j-1)}{2}\rceil}=\omega_j$};
\end{tikzpicture}$$

Hence, in general, the isomorphism $\Xi$ assumed in {\bf
 Step 1} doesn't exist, i.e., $F_2\ncong F_1$. By \cite[Theorem 9.1--9.3]{DG}, $$(\underline{\mathrm{K}}(F_1),\underline{\mathrm{K}}(F_1)_+)_{\Lambda}\ncong (\underline{\mathrm{K}}(F_2),\underline{\mathrm{K}}(F_2)_+)_{\Lambda}.$$

\end{proof}

\section{The map $\rho$ is necessary for extension}
In this section, we raise an example to show that  $\rho$ is necessary for extension, i.e., Theorem \ref{rho is nece} is true. The construction shares a similar process with the one described in section 2 of \cite{ALcounter}.
\begin{theorem}[Theorem 3.8 of \cite{ALcounter}; see also \cite{ERR2}]\label{yibande k0+}
Let $A,B$ be nuclear, separable ${\rm C}^*$-algebras of stable rank one and real rank zero. Suppose that $A$ is unital simple,  $B$ is stable and $({\rm K}_0(B),{\rm K}_0^+(B))$  is weakly unperforated. Let $e$ be a unital essential extension with trivial boundary maps
$$
0\to B\xrightarrow{\iota} E\xrightarrow{\pi} A\to 0.
$$
Then
$$\mathrm{K}_0^+(E)=((\mathrm{K}_0(\pi))^{-1} (\mathrm{K}_0^+(A)\backslash\{0\}))
\,\cup\,  (\mathrm{K}_0(\iota)(\mathrm{K}_0^+(B))).$$

\end{theorem}

\begin{example}\label{rho example}

Set $B_1=F_1$, $B_2=F_2$, where $F_1,F_2$ are the algebras in \ref{def F1F2}.
Denote $\mathbf{Z}$ to be the subgroup of $\mathbb{Z}[\frac{1}{4}]\oplus\prod_{{\mathbb{Z}\backslash\{0\}}} \mathbb{Z}[\frac{1}{4}]$
consisting of all elements of form $(a, \prod_{{\mathbb{Z}\backslash\{0\}}} a_j)$ and $ a_j= (|j|!)\cdot a$  for
all large enough $j$. Then we have
$$
\mathrm{K}_0(B_i)=\mathbf{Z},\,\,\,\, \mathrm{K}_1(B_i)\cong \mathbb{Z}_3,\,\,\,\,
 \mathrm{K}_0^+(B_i)=\mathbf{Z}\cap ( \mathbb{R}_+ \oplus \prod_{{\mathbb{Z}\backslash\{0\}}}\mathbb{R}_+),
$$
where $i=1,2$ and $\mathbb{R}_+=[0,+\infty)$.

Denote $\mathbf{Q}$ to be the subgroup of $\mathbb{Q}\oplus\prod_{\mathbb{Z}\backslash\{0\}}\mathbb{Q}$ consists of all the elements with form $(a, \prod_{{\mathbb{Z}\backslash\{0\}}} a_j)$ and $ a_j= (l_{|j|!})\cdot a$  for
all large enough $j$.

Consider the following exact sequence
$$0\to\mathbf{Z}\xrightarrow{\iota}\mathbf{Q}\xrightarrow{\varpi} \mathbf{Q}/\mathbf{Z}\to 0,$$
where $\iota(a, \prod_{\mathbb{Z}\backslash\{0\}}a_j)=(a, \prod_{\mathbb{Z}\backslash\{0\}}a_j)$
is the natural embedding map
and $\varpi$ is the quotient map. Then $\mathbf{Q}/\mathbf{Z}$ is a countable torsion group.
By \cite[Theorem 4.20]{EG}, let $A_\chi$ be a unital, simple, real rank zero and stable rank one AH algebra with no dimension growth satisfying
 $$\mathrm{K}_0(A_\chi)=\mathbb{Q}\oplus \mathbf{Q}/\mathbf{Z},\quad \mathrm{K}_1(A_\chi)=0,\quad [1_{A_\chi}]=[(1,0)] $$
 $$\mathrm{K}_0^+(A_\chi)=\{(x,y)\in \mathbb{Q}\oplus \mathbf{Q}/\mathbf{Z}\mid\,x>0\}\cup\{(0,0)\in \mathbb{Q}\oplus \mathbf{Q}/\mathbf{Z}\}.$$
Let $\varepsilon_i\in {\rm Ext}_{[1]}(\mathrm{K}_*(A_\chi), \mathrm{K}_*(B_i))$ be the equivalent class of
$$
0\to \mathrm{K}_0(B_i)\xrightarrow{\rho_i} (\mathbb{Q}\oplus \mathbf{Q},(1,0))\xrightarrow{\varpi_i} (\mathrm{K}_0(A_\chi),(1,0))\to 0,
$$
$$
0\to \mathrm{K}_1(B_i)\xrightarrow{\nu_i}\mathbb{Z}_3\xrightarrow{0} \mathrm{K}_1(A_\chi)\to 0,
$$
where $\rho_i(a, \prod_{\mathbb{Z}\backslash\{0\}}a_j)=(0, (a, \prod_{\mathbb{Z}\backslash\{0\}}a_j))$, $\nu_i=id_{\mathbb{Z}_3}$ and
$$\varpi_i(x, (b, \prod_{\mathbb{Z}\backslash\{0\}}b_j))=(x, \varpi(b, \prod_{\mathbb{Z}\backslash\{0\}}b_j)).$$

By Theorem \ref{strong wei}, there are two unital essential extensions of $C^*$-algebras
$$
0\to B_i\xrightarrow{\iota_i} E_i\xrightarrow{\pi_i}A_\chi\to 0, \quad i=1,2,
$$
realising $\varepsilon_i$ 
and satisfying the following commutative diagram
$$
\xymatrixcolsep{1pc}
\xymatrix{
{\,\,0\,\,} \ar[r]^-{}
& {\,\,\mathrm{K}_*(B_i)\,\,} \ar[d]_-{id} \ar[r]^-{(\rho_i,\nu_i)}
& {\,\,(\mathbb{Q}\oplus \mathbf{Q},\mathbb{Z}_3,(1,0))\,\,} \ar[d]_-{\zeta_i} \ar[r]^-{(\varpi_i,0)}
& {\,\,(\mathrm{K}_*(A_\chi), [1_{A_\chi}])\,\,} \ar[d]_-{id} \ar[r]^-{}
& {\,\,0\,\,} \\
{\,\,0\,\,} \ar[r]^-{}
& {\,\,\mathrm{K}_*(B_i)\,\,} \ar[r]^-{\mathrm{K}_*(\iota_i)}
& {\,\,(\mathrm{K}_*(E_i),[1_{E_i}]) \,\,} \ar[r]^-{\mathrm{K}_*(\pi_i)}
& {\,\,(\mathrm{K}_*(A_\chi), [1_{A_\chi}]) \,\,} \ar[r]_-{}
& {\,\,0\,\,}.}
$$
Then $E_1, E_2\in \mathcal{E}$. By definition of strongly unitary equivalence, for each $i=1,2$, $E_i$ is unique up to isomorphism.
We will identify
$\mathrm{K}_*(E_i)=(\mathbb{Q}\oplus \mathbf{Q},\mathbb{Z}_3)$ through the isomorphism $\zeta_i$.

\begin{proposition}
For each $i=1,2$ and any $n\geq1$, we have

{\rm (i)} if $3\mid n$, $
\mathrm{K}_0(E_i;\mathbb{Z}_{n})=\mathbb{Z}_{3},$ $
\mathrm{K}_1(E_i;\mathbb{Z}_{n})=\mathbb{Z}_{3};
$

{\rm (ii)} if $3\nmid n$, $\mathrm{K}_0(E_i;\mathbb{Z}_{n})=0,$ $
\mathrm{K}_1(E_i;\mathbb{Z}_{n})=0.$

\end{proposition}
\begin{proof}
For each $i=1,2$, we have
$$
\mathrm{K}_0(E_i)=\mathbb{Q}\oplus \mathbf{Q}\quad{\rm and}\quad
\mathrm{K}_1(E_i)=\mathbb{Z}_3.$$
For each $n\geq1$,
via  the cofibre sequence, we have the following exact sequence
$$
\xymatrixcolsep{3pc}
\xymatrix{
{\mathrm{K}_0(E_i)}  \ar[r]^-{\times n}
& {\mathrm{K}_0(E_i)}  \ar[r]^-{}
& {\mathrm{K}_0(E_i; \mathbb{Z}_n)} \ar[d]_-{}
 \\
{\mathrm{K}_1(E_i; \mathbb{Z}_n)} \ar[u]_-{}
& {\mathrm{K}_1(E_i)} \ar[l]_-{}
& {\mathrm{K}_1(E_i)} \ar[l]_-{\times n}.
}
$$
Note that the map
$\times n:\,\mathbb{Q}\oplus \mathbf{Q}\to \mathbb{Q}\oplus \mathbf{Q}$ in the first row is an isomorphism.
If $3\mid n$, the map $\times n:\,\mathbb{Z}_3\to \mathbb{Z}_3$ in the second row is the zero map; if $3\nmid n$,  the map $\times n:\,\mathbb{Z}_3\to \mathbb{Z}_3$ is an isomorphism. From the exactness of the diagram, the conclusions  are obvious.

\end{proof}

\begin{notion}\label{eta map of rho}\rm
By  Theorem \ref{thm mod DE}, 
$\gamma$ is a $\underline{\mathrm{K}}_{\langle\rho\rangle}$-isomorphism
$$(\underline{\mathrm{K}}(B_1),\underline{\mathrm{K}}
(B_1)_+))_{\underline{\mathrm{K}}_{\langle\rho\rangle}}
\cong (\underline{\mathrm{K}}(B_2),\underline{\mathrm{K}}(B_2)_+)_{
\underline{\mathrm{K}}_{\langle\rho\rangle}}.
$$
which is not $\Lambda$-linear (this map does not commute with some $\rho$ maps).
Thus, the restriction from
$\underline{\mathrm{K}}(B_1)_+$ to $\underline{\mathrm{K}}(B_2)_+$
forms a graded bijection,
That is,
$$\gamma:\,(\underline{\mathrm{K}}(B_1),\underline{\mathrm{K}}(B_1)_+)
\to (\underline{\mathrm{K}}(B_2),\underline{\mathrm{K}}(B_2)_+),
$$
is just a graded ordered isomorphism, 
at the same time, we have the restriction
$$
\gamma|_{{\mathrm{K}_*}(B_1)
\to {\mathrm{K}_*}(B_2)}=id:\,{\mathrm{K}_*}(B_1)
\to {\mathrm{K}_*}(B_2).
$$
Moreover, the induced map of $\gamma$ coincides with the identity map $$id:\underline{\mathrm{K}}(B_1|\,E_1)
\to \underline{\mathrm{K}}(B_2|\,E_2).$$
Thus, the following diagram is commutative
$$
\xymatrixcolsep{2pc}
\xymatrix{
{\,\,\underline{\mathrm{K}}(B_1)\,\,} \ar[r]^-{\underline{\mathrm{K}}(\iota_1)}\ar[d]_-{\gamma}
& {\,\,\underline{\mathrm{K}}(E_1)\,\,} \ar[d]_-{ \eta\triangleq id}
 \\
{\,\,\underline{\mathrm{K}}(B_2) \,\,}\ar[r]_-{\underline{\mathrm{K}}(\iota_2)}
& {\,\,\underline{\mathrm{K}}(E_2)\,\,}.}
$$

\end{notion}

Now we have the following theorem, which implies Theorem \ref{rho is nece}. The proof is similar to the one in \cite[Theorem 3.10]{ALcounter}.
\begin{theorem}\label{rho nec thm}
   There exists  a $\Lambda$-isomorphism $$
\eta\triangleq id:(\underline{\mathrm{K}}(E_1),\underline{\mathrm{K}}(E_1)_+,[1_{E_1}])_{\Lambda}\cong
(\underline{\mathrm{K}}(E_2),\underline{\mathrm{K}}(E_2)_+,[1_{E_2}])_{\Lambda},
$$
while $E_1\ncong E_2$.
\end{theorem}
\begin{proof}
By Theorem \ref{yibande k0+}, 
for each $i=1,2$, we have
$$\mathrm{K}_0^+(E_i)=\{(\frac{l}{k},y)|\,k,l\in\mathbb{N}\backslash\{0\}, y\in \mathbf{Q}\}\cup \{(0,y)|\,y\in \mathrm{K}_0^+(B_i)\},$$
then $\mathrm{K}_0^+(E_1)\cong \mathrm{K}_0^+(E_2).$

Consider the sequence
$$
\underline{\mathrm{K}}(B_i)\xrightarrow{\underline{\mathrm{K}}(\iota_i)} \underline{\mathrm{K}}(E_i)\to \underline{\mathrm{K}}(A_\chi),
$$
where $\underline{\mathrm{K}}(\iota_i)$ is not injective, though ${\mathrm{K}_*}(\iota_i)$ is injective.
Then $\iota_i$ will induce the following commutative diagram with exact rows
$$
\xymatrixcolsep{1pc}
\xymatrix{
{\mathrm{K}_0(B_i)} \ar[d]^-{{\rm K}_0(\iota_i)} \ar[r]^-{\times n}
& {\mathrm{K}_0(B_i)} \ar[d]^-{{\rm K}_0(\iota_i)} \ar[r]^-{}
& {\mathrm{K}_0(B_i;\mathbb{Z}_n)} \ar[d]^-{{\rm K}_0(\iota_i;\mathbb{Z}_n)} \ar[r]^-{}
& 
 {\mathrm{K}_1(B_i)} \ar[d]^-{{\rm K}_1(\iota_i)} \ar[r]^-{\times n}
& {\mathrm{K}_1(B_i)} \ar[d]^-{{\rm K}_1(\iota_i)} \ar[r]^-{}
& {\mathrm{K}_1(B_i;\mathbb{Z}_n)} \ar[d]^-{{\rm K}_1(\iota_i;\mathbb{Z}_n)}
 \\
{\mathrm{K}_0(E_i)} \ar[r]^-{\times n}
& {\mathrm{K}_0(E_i)} \ar[r]_-{}
& {\mathrm{K}_0(E_i; \mathbb{Z}_n)} \ar[r]^-{}
& 
{\mathrm{K}_{1}(E_i)} \ar[r]^-{\times n}
& {\mathrm{K}_{1}(E_i)}\ar[r]^-{}
& {\mathrm{K}_1(B_i;\mathbb{Z}_n)} ,
}
$$
which is
$$
\xymatrixcolsep{2pc}
\xymatrix{
{\mathbf{Z}} \ar[d]^-{{\rm K}_0(\iota_i)} \ar[r]^-{\times n}
& {\mathbf{Z}} \ar[d]^-{{\rm K}_0(\iota_i)} \ar[r]^-{}
& {\mathrm{K}_0(B_i; \mathbb{Z}_n)} \ar[d]^-{{\rm K}_0(\iota_i;\mathbb{Z}_n)} \ar[r]^-{(\beta_{B_i})_n^{0}}
& 
 {\mathbb{Z}_3} \ar[d]^-{{\rm id}} \ar[r]^-{\times n}
& {\mathbb{Z}_3} \ar[d]^-{{\rm id}}\ar[r]^-{(\rho_{B_i})_n^{1}}
& {\mathrm{K}_1(B_i; \mathbb{Z}_n)} \ar[d]^-{{\rm K}_1(\iota_i;\mathbb{Z}_n)}
 \\
{\mathbb{Q}\oplus \mathbf{Q}} \ar[r]^-{\times n}
& {\mathbb{Q}\oplus \mathbf{Q}} \ar[r]_-{}
& {\mathrm{K}_0(E_i; \mathbb{Z}_n)} \ar[r]^-{(\beta_{E_i})_n^{0}}
& 
{\mathbb{Z}_3} \ar[r]^-{\times n}
& {\mathbb{Z}_3} \ar[r]^-{(\rho_{E_i})_n^{1}}
& {\mathrm{K}_1(E_i; \mathbb{Z}_n)}.
}
$$
(Recall the construction, we take ${\rm K}_1(\iota_i)={\rm id}_{\mathbb{Z}_3}$.)
In the above diagram: when $n\nmid3$,
we have $\mathrm{K}_0(E_i; \mathbb{Z}_n)=0$ and $\mathrm{K}_1(E_i; \mathbb{Z}_n)=0$.
So both ${\rm K}_0(\iota_i;\mathbb{Z}_n)$ and ${\rm K}_1(\iota_i;\mathbb{Z}_n)$ are automatically surjective.
When $n\mid 3$ ($n\neq 1$), the map
$\times n:\,\mathbb{Q}\oplus \mathbf{Q}\to\mathbb{Q}\oplus \mathbf{Q} $ is an isomorphism,
$\times n:\,\mathbb{Z}_3\to \mathbb{Z}_3$ is the zero map and
$\times n:\,\mathbf{Z}\to \mathbf{Z} $ is injective. By exactness of both the rows,
we have
$(\beta_{E_i})_n^{0},(\rho_{E_i})_n^{1},(\rho_{B_i})_n^{1}$ are all  isomorphisms and $(\beta_{B_i})_n^{0}$ is surjective.
By commutativity, we have both ${\rm K}_0(\iota_i;\mathbb{Z}_n)$ and ${\rm K}_1(\iota_i;\mathbb{Z}_n)$ are surjective.

Thus, for each $i=0,1$, we have $(x,\mathfrak{u},\bigoplus_{n=1}^{\infty}  ( \mathfrak{s}_{n,0}, \mathfrak{s}_{n,1}))\in  \underline{{\rm K}}(E_i)_+$
is positive if and only if one of the following conditions is satisfied:

(1) $x=(\frac{l}{k},y)$ and $\frac{l}{k}>0$;

(2) $x=(0,y)$, $y=(a,\prod_{\mathbb{Z}} a_m)\in {\rm K}_0^+(B_i)$ and $a>0$;

(3) $x=(0,y)$, $y=(0,\prod_{\mathbb{Z}} a_m)\in {\rm K}_0^+(B_i)$, $\mathfrak{u}=0,$ $\mathfrak{s}_{n,0}=0,$ $\mathfrak{s}_{n,1}=0$ for all  $n\geq1$.

That is, $\eta$ is an order preserving map. It is routine to check that $\eta$ is $\Lambda$-linear.

As
$B_1$, $B_2$ are the unique maximal proper ideals of $E_1$, $E_2$, respectively,
$B_1\ncong B_2$ implies
$E_1\ncong E_2.$

\end{proof}

\end{example}

\section{The map $\beta$ is automatic for extension}
 Theorem \ref{beta not ness} is a special case of the following theorem.
\begin{theorem}\label{beta aut thm}
Let $A_1, A_2, B_1,B_2$ be $C^*$-algebras.
Given two  extensions of $C^*$-algebras with trivial boundary maps
$$
0\to B_i \xrightarrow{\iota_i} E_i\to A_i\to 0,\quad i=1,2.
$$
Suppose we have two graded group isomorphisms  $$\gamma:\, \underline{\mathrm{K}}(B_1)\to\underline{\mathrm{K}}(B_2)\quad
{and}\quad
\eta:\, \underline{\mathrm{K}}(E_1)\to\underline{\mathrm{K}}(E_2),$$
such that the following diagram is commutative at the level of graded groups
$$
\xymatrixcolsep{2pc}
\xymatrix{
{\,\,\underline{\mathrm{K}}(B_1)\,\,} \ar[r]^-{\underline{\mathrm{K}}(\iota_1)}\ar[d]_-{\gamma}
& {\,\,\underline{\mathrm{K}}(E_1)\,\,} \ar[d]_-{ \eta}
 \\
{\,\,\underline{\mathrm{K}}(B_2) \,\,}\ar[r]_-{\underline{\mathrm{K}}(\iota_2)}
& {\,\,\underline{\mathrm{K}}(E_2)\,\,}.}
$$
If $\eta$ preserves the action of $\beta$, then $\gamma$ preserves the action of $\beta$, automatically.
\end{theorem}
\begin{proof}
Given $n\geq 1$, we will show the following diagram is commutative.
\begin{displaymath}
\xymatrixcolsep{0.4pc}
\xymatrix{
\mathrm{K}_0(B_1;\mathbb{Z}_n) \ar[ddd]^-{\gamma_n^0} \ar[rrr]^-{(\beta_{B_1})_n^{0}}\ar[dr]^-{\mathrm{K}_0(\iota_1;\mathbb{Z}_n)} &&& \mathrm{K}_1(B_1) \ar[ddd]^-{\gamma_0^1}  \ar[dr]^-{\mathrm{K}_1(\iota_1)}\\
&\mathrm{K}_0(E_1;\mathbb{Z}_n) \ar[rrr]^-{(\beta_{E_1})_n^{0}}\ar[d]^-{\eta_n^0}&&&
\mathrm{K}_1(E_1)   \ar[d]^-{\eta_0^1}\\
&\mathrm{K}_0(E_2;\mathbb{Z}_n)  \ar[rrr]_-{(\beta_{E_2})_n^{0}}&&&
\mathrm{K}_1(E_2).    \\
\mathrm{K}_0(B_2;\mathbb{Z}_n)\ar[rrr]_-{(\beta_{B_2})_n^{0}}  \ar[ur]_-{\mathrm{K}_0(\iota_2;\mathbb{Z}_n)} &&&\mathrm{K}_1(B_2) \ar[ur]_-{\mathrm{K}_1(\iota_2)}
}
\end{displaymath}
In fact, we only need to
prove the commutativity of the  following diagram:
$$
\xymatrixcolsep{2pc}
\xymatrix{
{\,\,\mathrm{K}_0(B_1;\mathbb{Z}_n)\,\,} \ar[r]^-{(\beta_{B_1})_n^{0}}\ar[d]_-{\gamma_n^0}
& {\,\,\mathrm{K}_1(B_1)\,\,} \ar[d]_-{ \gamma_0^1}
 \\
{\,\,\mathrm{K}_0(B_2;\mathbb{Z}_n) \,\,}\ar[r]_-{(\beta_{B_2})_n^{0}}
& {\,\,\mathrm{K}_1(B_2).\,\,}}
$$
For any $x_1\in \mathrm{K}_0(B_1;\mathbb{Z}_n)$, and denote
$$x_2=\gamma_n^0(x_1)\in \mathrm{K}_0(B_2;\mathbb{Z}_n),$$$$ y_1=(\beta_{B_1})_n^{0}(x_1)\in \mathrm{K}_1(B_1),\,\,\,\, y_2=(\beta_{B_2})_n^{0}(x_2)\in \mathrm{K}_1(B_2).$$
We are going to show that
$$\gamma_0^1(y_1)=y_2\in \mathrm{K}_1(B_2).$$

For $i=1,2$, denote
$$w_i=\mathrm{K}_0(\iota_i;\mathbb{Z}_n)(x_i)\in \mathrm{K}_0(E_i;\mathbb{Z}_n),
\,\,\,\, z_i=\mathrm{K}_1(\iota_i)(y_i)\in \mathrm{K}_1(E_i).$$

Since the both the induced maps $\underline{\mathrm{K}}(\iota_1)$ and $\underline{\mathrm{K}}(\iota_2)$ are $\Lambda$-linear, which certainly preserve the action of $\beta$, for each integer $n\geq 1$,
we have both
$$
\xymatrixcolsep{2pc}
\xymatrix{
{\,\,\mathrm{K}_0(B_1;\mathbb{Z}_n)\,\,} \ar[r]^-{(\beta_{B_1})_n^{0}}\ar[d]_-{\mathrm{K}_0(\iota_1;\mathbb{Z}_n)}
& {\,\,\mathrm{K}_1(B_1)\,\,} \ar[d]_-{ \mathrm{K}_1(\iota_1)}
 \\
{\,\,\mathrm{K}_0(E_1;\mathbb{Z}_n) \,\,}\ar[r]_-{(\beta_{E_1})_n^{0}}
& {\,\,\mathrm{K}_1(E_1)\,\,}}
\quad
{\rm and}
\quad
\xymatrixcolsep{2pc}
\xymatrix{
{\,\,\mathrm{K}_0(B_2;\mathbb{Z}_n)\,\,} \ar[r]^-{(\beta_{B_2})_n^{0}}\ar[d]_-{\mathrm{K}_0(\iota_2;\mathbb{Z}_n)}
& {\,\,\mathrm{K}_1(B_2)\,\,} \ar[d]_-{ \mathrm{K}_1(\iota_2)}
 \\
{\,\,\mathrm{K}_0(E_2;\mathbb{Z}_n) \,\,}\ar[r]_-{(\beta_{E_2})_n^{0}}
& {\,\,\mathrm{K}_1(E_2)\,\,}}
$$
are commutative diagrams. This implies
$$
(\beta_{E_i})_n^{0}(w_i)=z_i\in \mathrm{K}_1(E_i),\quad i=1,2.
$$

By assumption, the following diagram is commutative
$$
\xymatrixcolsep{2pc}
\xymatrix{
{\,\,\mathrm{K}_0(B_1;\mathbb{Z}_n)\,\,} \ar[r]^-{\mathrm{K}_0(\iota_1;\mathbb{Z}_n)}\ar[d]_-{\gamma_n^0}
& {\,\,\mathrm{K}_0(E_1;\mathbb{Z}_n)\,\,} \ar[d]_-{\eta_n^0}
 \\
{\,\,\mathrm{K}_0(B_2;\mathbb{Z}_n) \,\,}\ar[r]_-{\mathrm{K}_0(\iota_2;\mathbb{Z}_n)}
& {\,\,\mathrm{K}_0(E_2;\mathbb{Z}_n)\,\,}.}
$$
This implies
$\eta_n^0(w_1)=w_2\in \mathrm{K}_0(E_2;\mathbb{Z}_n).
$
Since $\eta$ preserves the action of $\beta$,
then we have the following commutative diagram
$$
\xymatrixcolsep{2pc}
\xymatrix{
{\,\,\mathrm{K}_0(E_1;\mathbb{Z}_n)\,\,} \ar[r]^-{(\beta_{E_1})_n^{0}}\ar[d]_-{\eta_n^0}
& {\,\,\mathrm{K}_1(E_1)\,\,} \ar[d]_-{ \eta_0^1}
 \\
{\,\,\mathrm{K}_0(E_2;\mathbb{Z}_n) \,\,}\ar[r]_-{(\beta_{E_2})_n^{0}}
& {\,\,\mathrm{K}_1(E_2)\,\,}.}
$$
Then we obtain
$\eta_0^1(z_1)=z_2$, i.e.,
$$
\eta_0^1\circ \mathrm{K}_1(\iota_1)(y_1)=\mathrm{K}_1(\iota_2)(y_2)\in \mathrm{K}_1(E_2).
$$

Since the extensions we concern have trivial boundary maps, both $\mathrm{K}_1(\iota_1)$ and $\mathrm{K}_1(\iota_2)$ are injective. Then  the following  diagram
$$
\xymatrixcolsep{2pc}
\xymatrix{
{\,\,\mathrm{K}_1(B_1)\,\,} \ar[r]^-{\mathrm{K}_1(\iota_1)}\ar[d]_-{\gamma_0^1}
& {\,\,\mathrm{K}_1(E_1)\,\,} \ar[d]_-{ \eta_0^1}
 \\
{\,\,\mathrm{K}_1(B_2) \,\,}\ar[r]_-{\mathrm{K}_1(\iota_2)}
& {\,\,\mathrm{K}_1(E_2)\,\,}}
$$
is commutative and we do obtain
$$\gamma_0^1(y_1)=y_2\in\mathrm{K}_1(B_2) .$$

Similarly, one can show that
the following diagram is also commutative
$$
\xymatrixcolsep{2pc}
\xymatrix{
{\,\,\mathrm{K}_1(B_1;\mathbb{Z}_n)\,\,} \ar[r]^-{(\beta_{B_1})_n^{1}}\ar[d]_-{\gamma_n^1}
& {\,\,\mathrm{K}_0(B_1)\,\,} \ar[d]_-{ \gamma_0^0}
 \\
{\,\,\mathrm{K}_1(B_2;\mathbb{Z}_n) \,\,}\ar[r]_-{(\beta_{B_2})_n^{1}}
& {\,\,\mathrm{K}_0(B_2)\,\,}.}
$$
\end{proof}
\begin{remark}
Note that in \cite[Theorem 3.3]{DE} (for convenience, we only consider the case of $n=3$), there exists two real rank zero non-isomorphic AD algebras $E,E'$ with
$$
(\underline{\mathrm{K}}(E),\underline{\mathrm{K}}(E)_+)_{\Lambda}\ncong (\underline{\mathrm{K}}(E'),\underline{\mathrm{K}}(E')_+)_{\Lambda}
$$
as $\Lambda$-modules, but there is an ordered $\underline{\mathrm{K}}_{\langle\beta\rangle}$-isomorphism
$$
\gamma:(\underline{\mathrm{K}}(E),
\underline{\mathrm{K}}(E)_+)_{\underline{\mathrm{K}}_{\langle\beta\rangle}}\cong (\underline{\mathrm{K}}(E'),
\underline{\mathrm{K}}(E')_+)_{\underline{\mathrm{K}}_{\langle\beta\rangle}}.
$$

Take new $B_1=E\otimes \mathcal{K}, B_2=E'\otimes \mathcal{K}$. We provide a similar construction of Example \ref{rho example} as follows:

Denote $\mathrm{K}_0(B_i)$ by $\widetilde{\mathbf{Z}}$  and set $\widetilde{\mathbf{Q}}=\widetilde{\mathbf{Z}}\otimes \mathbb{Q}$.
Consider the following exact sequence
$$0\to\widetilde{\mathbf{Z}}\xrightarrow{\iota}\widetilde{\mathbf{Q}}\xrightarrow{\varpi} \widetilde{\mathbf{Q}}/\widetilde{\mathbf{Z}}\to 0,$$
where $\iota$ is the natural embedding map
and $\varpi$ is the quotient map. 
By \cite[Theorem 4.20]{EG}, let $\widetilde{A}_\chi$ be a unital, simple, real rank zero and stable rank one AH algebra with no dimension growth satisfying
 $$\mathrm{K}_0(\widetilde{A}_\chi)=\mathbb{Q}\oplus \widetilde{\mathbf{Q}}/\widetilde{\mathbf{Z}},\quad \mathrm{K}_1(\widetilde{A}_\chi)=0,\quad [1_{\widetilde{A}_\chi}]=[(1,0)] $$
 $$\mathrm{K}_0^+(\widetilde{A}_\chi)=\{(x,y)\in \mathbb{Q}\oplus \widetilde{\mathbf{Q}}/\widetilde{\mathbf{Z}}\mid\,x>0\}\cup\{(0,0)\in \mathbb{Q}\oplus \widetilde{\mathbf{Q}}/\widetilde{\mathbf{Z}}\}.$$
Let
$$
0\to B_i\xrightarrow{\iota_i} E_i\xrightarrow{\pi_i}\widetilde{A}_\chi\to 0, \quad i=1,2,
$$
be the unital extensions which realise the classes of
$$
0\to \mathrm{K}_0(B_i)\xrightarrow{\rho_i} (\mathbb{Q}\oplus \widetilde{\mathbf{Q}},(1,0))\xrightarrow{\varpi_i} (\mathrm{K}_0(\widetilde{A}_\chi),(1,0))\to 0,
$$
$$
0\to \mathrm{K}_1(B_i)\xrightarrow{id}\mathrm{K}_1(B_i)\xrightarrow{0} \mathrm{K}_1(\widetilde{A}_\chi)\to 0.
$$

Following a similar discussion of \ref{eta map of rho}, one can construct a graded group isomorphism $\eta:\,\underline{\mathrm{K}}(E_1)\to \underline{\mathrm{K}}(E_2)$ with
a commutative diagram
$$
\xymatrixcolsep{2pc}
\xymatrix{
{\,\,\underline{\mathrm{K}}(B_1)\,\,} \ar[r]^-{\underline{\mathrm{K}}(\iota_1)}\ar[d]_-{\gamma}
& {\,\,\underline{\mathrm{K}}(E_1)\,\,} \ar[d]_-{ \eta}
 \\
{\,\,\underline{\mathrm{K}}(B_2) \,\,}\ar[r]_-{\underline{\mathrm{K}}(\iota_2)}
& {\,\,\underline{\mathrm{K}}(E_2)\,\,}.}
$$
But the map $\eta:\,\underline{\mathrm{K}}(E_1)\to \underline{\mathrm{K}}(E_2)$ can at most be an ordered $\underline{\mathrm{K}}_{\langle\beta\rangle}$-isomorphism rather than an ordered $\Lambda$-isomorphism. This is in contrast with Theorem \ref{rho nec thm} and Theorem \ref{beta aut thm}.

Moreover, we point out that if we don't insist the requirement of the commutativity of total K-theory, we still have an ordered scaled $\Lambda$-isomorphism  $$\zeta:\,
(\underline{\rm K}(E_1),\underline{\rm K}(E_1)_+,[1_{E_1}])_{\Lambda}
\cong  (\underline{\rm K}(E_2),\underline{\rm K}(E_2)_+,[1_{E_2}])_{\Lambda}$$
together with the following commutative diagram of K$_*$-groups
$$
\xymatrixcolsep{2pc}
\xymatrix{
{\,\,{\mathrm{K}}_*(B_1)\,\,} \ar[r]^-{{\mathrm{K}}_*(\iota_1)}\ar[d]_-{\gamma_*}
& {\,\,{\mathrm{K}}_*(E_1)\,\,} \ar[d]_-{ \zeta_*}
 \\
{\,\,{\mathrm{K}}_*(B_2) \,\,}\ar[r]_-{{\mathrm{K}}_*(\iota_2)}
& {\,\,{\mathrm{K}}_*(E_2)\,\,}.}
$$
This concludes Theorem \ref{beta+com}. Here,  the construction of $\zeta$ can be as  below:

Note that
$${\mathrm{K}}_0(E_i)\cong\mathbb{Q}\oplus \widetilde{\mathbf{Q}}/\widetilde{\mathbf{Z}}
\quad {\rm and}\quad {\mathrm{K}}_1(E_i)\cong{\mathrm{K}}_1(B_i),\quad i=1,2.$$
Concretely, for each $i=1,2,$
${\mathrm{K}}_1(E_i)$ consists of all the elements of the form
$(a,(a_m))\in \mathbb{Z}\oplus \prod_{m\geq 1}\mathbb{Z}_3$ with $a_m=(-1)^{i-1}\cdot[a]_3$ for all large enough $m$. Then from the following exact sequence
$$
\xymatrixcolsep{3pc}
\xymatrix{
{\mathrm{K}_0(E_i)}  \ar[r]^-{\times n}
& {\mathrm{K}_0(E_i)}  \ar[r]^-{}
& {\mathrm{K}_0(E_i; \mathbb{Z}_n)} \ar[d]_-{}
 \\
{\mathrm{K}_1(E_i; \mathbb{Z}_n)} \ar[u]_-{}
& {\mathrm{K}_1(E_i)} \ar[l]_-{}
& {\mathrm{K}_1(E_i)} \ar[l]_-{\times n},
}
$$ we have
$$
{\mathrm{K}}_0(E_i;\mathbb{Z}_n)\cong
\begin{cases}
  \bigoplus_{n\geq 1}\mathbb{Z}_3, & \mbox{if } n\mid3 \\
   0, & \mbox{if } n\nmid 3
\end{cases}
$$
and
${\mathrm{K}}_1(E_i;\mathbb{Z}_n)\cong{\mathrm{K}}_1(E_i)/ (n\cdot{\mathrm{K}}_1(E_i))$ whose elements has the form
$(b,(b_m))\in \mathbb{Z}_n\oplus \prod_{m\geq 1}(\mathbb{Z}_3/(n\cdot\mathbb{Z}_3))$.

We take
$$
\zeta_0^0=id:\,{\mathrm{K}}_0(E_1)\to {\mathrm{K}}_0(E_2),
$$
$$
\zeta_0^1:\,{\mathrm{K}}_1(E_1)\to {\mathrm{K}}_1(E_2)\quad{\rm with}\quad \zeta_0^1(a,(a_m))=(a,(-a_m));
$$
for $n\geq 2$, we take
$$
\zeta_n^0=-id:\,{\mathrm{K}}_0(E_1;\mathbb{Z}_n)\to {\mathrm{K}}_0(E_2;\mathbb{Z}_n),
$$
$$
\zeta_n^1={\mathrm{K}}_1(E_1;\mathbb{Z}_n)\to {\mathrm{K}}_1(E_2;\mathbb{Z}_n)\quad{\rm with}\quad \zeta_n^1(b,(b_m))=(b,(-b_m)).
$$

\end{remark}


\begin{thebibliography}{}


\bibitem{AELL}
Q. An, G. A. Elliott, Z. Li and Z. Liu. The classification of certain ASH C*-algebras of real rank zero. J. Topol. Anal., 14 (1) (2022), 183--202.

\bibitem{ALjfa}
Q. An and Z. Liu. A total Cuntz semigroup for C*-algebras of stable rank one. J. Funct. Anal., 284 (8) (2023), No. 109858.

\bibitem{AL2023}
Q. An and Z. Liu. On the range of certain ASH algebras of real rank zero. Chinese Ann. Math. Ser. B., 44 (2023), no. 2, 279--288.

\bibitem{ALpams}
Q. An and Z. Liu. On unital absorbing extensions of C$^*$-algebras of stable rank one and real rank zero. Proc.
Amer. Math. Soc., 152 (6) (2024),  2497--2510.

\bibitem{ALcounter}
Q. An and Z. Liu. Total Cuntz semigroup, extension, and Elliott conjecture with real rank zero. Proc. Lond. Math. Soc. (3), 128 (2024), no. 4, Paper No. e12595, 40 pp.

\bibitem{ALL}
Q. An, C. Li and Z. Liu. A latticed total K-theory. ArXiv:2408.15941v1.


\bibitem{ALZ}
Q. An, Z. Liu and Y. Zhang.  On the classification of certain real rank zero C*-algebras. Sci. China Math., 65 (4) (2022), 753--792.

\bibitem {DE}
M. Dadarlat and S. Eilers. Compressing coefficients while preserving ideals in the
K-theory for C*-algebras. K-Theory, 14 (1998), 281--304.

\bibitem {DEb}
M. Dadarlat and S. Eilers. The Bockstein Map is Necessary. Canad. Math. Bull., 42 (3) (1999), 274--284.

\bibitem {DG}
M. Dadarlat and G. Gong. A classification result for approximately homogeneous C*-algebras of real rank zero. Geom. Funct. Anal., 7 (4) (1997), 646--711.


\bibitem{DL2}
M. Dadarlat and T. Loring. A universal multi-coefficient theorem for the Kasparov groups. Duke Math. J., 84 (1996), 355--377.

\bibitem {DL3}
M. Dadarlat and T. A. Loring. Classifying C*-algebras via ordered mod-p K-theory. Math. Ann., 305 (1996), 601--616.

\bibitem {Eiphd}
S. Eilers. Invariants for AD algebras. Ph.D. thesis, Copenhagen University, 1995.


\bibitem {Ei}
S. Eilers. A complete invariant for AD algebras with bounded torsion in $K_1$. J. Funct. Anal., 139 (2) (1996), 325--348.

\bibitem {ERR2}
S. Eilers, G. Restorff and  E. Ruiz. The ordered  K-theory of a full extension.
Canad. J. Math., 66 (3) (2014), 596--625.

\bibitem{ERZ}
S. Eilers, G. Restorff and E. Ruiz. Classification of extensions of classifiable
C*-algebras. Adv. Math., 222 (6) (2009), 2153--2172.


\bibitem {Ell}
G. A. Elliott. On the classification of C*-algebras of real rank zero. J.
Reine Angew. Math., 443 (1993), 179--219.


\bibitem{EG}
G. A. Elliott and G. Gong. On the classification of C*-algebras of real rank zero, II. Ann. of Math., 144 (2) (1996), 497--610.

\bibitem{EGS}
G. A. Elliott, G. Gong and H. Su. On the classification of C*-algebras of real rank zero. IV. Reduction to local spectrum of dimension two. Operator Algebras and their applications[C]// Fields Inst Commun. 1998.





\bibitem{GR}
J. Gabe and E. Ruiz. The unital Ext-groups and classification of C*-algebras. Glasgow Math. J., 62 (2020) 201--231.

\bibitem{G}
G. Gong. Classification of C*-Algebras of Real Rank Zero and Unsuspended E-Equivalence Types. J. Funct. Anal., 152 (2) (1998), 281--329.

\bibitem{GJL}
G. Gong, C. Jiang and L. Li. Hausdorffified algebraic $K_1$-group and invariants for C*-algebras with the ideal property. Ann. K-Theory, 5 (1) (2020), 43--78.

\bibitem{GJL2}
G. Gong, C. Jiang and L. Li. A classification of inductive limit C*-algebras with ideal property. Trans. London Math. Soc., 9 (1) (2022), 158--236.

\bibitem{GJLP1}
G. Gong, C. Jiang, L. Li and C. Pasnicu. A$\mathbb{T}$  structure of  AH  algebras with the ideal property and torsion free  K-theory.
J. Funct. Anal., 258 (6) (2010), 2119--2143.

\bibitem{GJLP2}
G. Gong, C. Jiang, L. Li and C. Pasnicu.  A reduction theorem for  AH  algebras with the ideal property.
Int. Math. Res. Not. IMRN, 24 (2018), 7606--7641.





\bibitem{LR}
H. Lin and M. R\o rdam. Extensions of inductive limits of circle algebras. J. London Math. Soc., 51 (2) (1995), 603--613.


\bibitem{RS}
J. Rosenberg and C. Schochet. The K\"{u}nneth theorem and the universal coefficient theorem for Kasparov's generalized $\mathrm{K}$-functor. Duke Math. J., 55 (1987), 431--474.



\bibitem{S}
C. Schochet. Topological methods for for C*-algebras IV: Mod $p$
K-theory. Pacific J. Math., 114 (1984), 447--468.

\end{thebibliography}
\end{document}